\documentclass{amsart}
\usepackage{amssymb}
\usepackage{latexsym}
\usepackage{amsmath}

\def\wt{\widetilde}
\def\wh{\widehat}
\def\ov{\overline}
\def \im{{\rm Im}}
 \def\up{\upharpoonright}
\def\cH{\mathcal H}

\def\cK{\mathcal K} \def\cL{\mathcal L}
 \def\cN{\mathcal N} \def\cP{\mathcal P} 
 \def\cT{\mathcal T} \def\cI{\mathcal I}

\def \gH{\mathfrak H}   \def \gN{\mathfrak N}

\def \bC{\mathbb C}    \def\bR{\mathbb R}
\def\bH{\mathbb H} 
\def \l{\lambda}
 \def \b{\beta}    
 \def \t{\theta} \def\g {\gamma}
  \def\om {\omega} \def\Om {\Omega} 
\def \f{\varphi}  \def \G{\Gamma} \def\D {\Delta} 

\def \C{\widetilde {\mathcal C}}
\def \CA{\C(\cH_0,\cH_1)}

\def \cd {\cdot}

\def\AC {AC(\cI; \bH)}   \def\LI {L_\Delta^2(\cI)}
\def\lI {\cL_\Delta^2(\cI)}

\def\tma{\cT_{\max}} \def\tmi{\cT_{\min}} \def\Tma{T_{\max}} \def\Tmi{T_{\min}}

\def \dom {{\rm dom}\,}  \def \ran {{\rm ran}\,}  \def \ker{{\rm ker\,}}
 \def \mul {{\rm mul}\,}

  \def\tm{\times}

\def  \RH {\wt R_+(\cH_0,\cH_1)}

\def \pair {\tau=\{\tau_+,\tau_-\}}
\def \CR {\bC\setminus\bR}

\newcommand {\lo}[1] {\cL_\D^2[#1,\bH ]}

\def\bt{\{\cH,\G_0,\G_1\}}
\def\bta{\{\cH_0\oplus \cH_1,\Gamma _0,\Gamma _1\}}

\newtheorem{theorem}{Theorem}[section]
\newtheorem{proposition}[theorem]{Proposition}
\newtheorem{corollary}[theorem]{Corollary}
\newtheorem{lemma}[theorem]{Lemma}

\theoremstyle{definition}

\theoremstyle{definition}
\newtheorem {definition} [theorem]{Definition}
\theoremstyle{remark}
\newtheorem{remark}[theorem]{Remark}
\numberwithin{equation}{section}
\begin{document}
\title[On generalized resolvents and characteristic matrices]
{On generalized resolvents and characteristic matrices of first-order symmetric systems }
\author {Vadim Mogilevskii}
\address{Institute of Applied Mathematics and Mechanics, NAS of
Ukraine,  R. Luxemburg Str. 74,  83050 Donetsk,    Ukraine}
\address{Department of Mathematics,  Lugans'k National University,
 Oboronna Str. 2, 91011  Lugans'k,   Ukraine}
\email{vim@mail.dsip.net}


\begin{abstract}
We study general (not necessarily Hamiltonian) first-order symmetric system $J y'-B(t)y=\D(t) f(t)$ on an interval $\cI=[a,b) $ with the regular endpoint $a$ and singular endpoint $b$. It is assumed that the deficiency indices $n_\pm(\Tmi)$ of the corresponding
minimal relation $\Tmi$ in $\LI$ satisfy $n_-(\Tmi)\leq n_+(\Tmi)$. We describe all generalized resolvents $y=R(\l)f, \; f\in\LI,$ of $\Tmi$ in terms of boundary problems with $\l$-depending boundary conditions imposed on regular and singular boundary values of a function $y$ at the endpoints $a$ and $b$ respectively. We also parametrize all characteristic matrices $\Om(\l)$ of the system immediately in terms of boundary conditions. Such a parametrization is given both  by the block representation  of $\Om(\l)$ and by the formula similar to the well-known Krein formula for resolvents. These results develop the \u{S}traus' results on generalized resolvents and characteristic matrices of differential operators.
\end{abstract}
\maketitle
\section{Introduction}
Let  $H$ and $\wh H$ be finite dimensional Hilbert spaces and let $\bH:=H\oplus\wh H \oplus H$. Denote also by $[\bH]$ the set of all linear
operators in $\bH$. We study first-order symmetric systems  of differential equations defined on an interval $\cI=[a,b), -\infty<a <b\leq\infty,$ with the
regular endpoint $a$ and regular or singular endpoint $b$. Such a system is of the form \cite{Atk,GK}
\begin {equation}\label{1.2}
J y'-B(t)y=\D(t) f(t), \quad t\in\cI,
\end{equation}
where $B(t)=B^*(t)$ and $\D(t)\geq 0$ are  $[\bH]$-valued functions on $\cI$ and
\begin {equation} \label{1.3}
J=\begin{pmatrix} 0 & 0&-I_H \cr 0& i I_{\wh H}&0\cr I_H&
0&0\end{pmatrix}:H\oplus\wh H\oplus H \to H\oplus\wh H\oplus H.
\end{equation}
With  \eqref{1.2} one associates the homogeneous system
%
\begin {equation}\label{1.5}
J y'-B(t)y=\l \D(t) y, \quad \l \in\bC.
\end{equation}

We assume that system \eqref{1.2} is definite (see Definition \ref{def3.1}). Recall also that system \eqref{1.2} is called a Hamiltonian system if $\wh H=\{0\}$ and hence
\begin {equation}\label{1.5a}
J=\begin{pmatrix}  0&-I_H \cr  I_H& 0\end{pmatrix}:H\oplus H \to H\oplus H.
\end{equation}

As  is known,  the extension theory of symmetric linear relations gives a natural framework   for investigation of the boundary value problems for symmetric systems (see \cite{BHSW10,DLS88,DLS93,Kac03,LesMal03,Orc} and references therein). According to \cite{Kac03,LesMal03, Orc}  the system \eqref{1.2} generates the minimal linear relation $\Tmi$ and the maximal
linear relation $\Tma$ in  the Hilbert space $\LI$ of all functions $f(\cd):\cI\to\bH$  satisfying $||f||_\D^2:= \int\limits_\cI
(\D(t)f(t),f(t))_\bH\,dt<\infty$. It turns out that $\Tmi$ is a closed symmetric relation and $\Tma=\Tmi^*$. Moreover, the  deficiency indices $n_\pm(\Tmi)$ of $\Tmi$ satisfy $\dim H\leq n_\pm(\Tmi)\leq \dim \bH$.

According to \cite{Bru78,DLS88,Sht57} each generalized  resolvent $R(\l)$ of $\Tmi$ admits the representation
\begin {equation*}
(R(\l) f)(x)=\int_\cI Y_0(x,\l)(\Om(\l)+\tfrac 1 2 \, {\rm
sgn}(t-x)J)Y_0^*(t,\ov\l)\D(t) f(t)\,dt , \quad f=f(\cd)\in\LI.
\end{equation*}
Here $Y_0(\cd,\l)$ is an $[\bH]$-valued operator solution of Eq.  \eqref{1.5} satisfying $_0(a,\l)=I_{\bH}$ and $\Om(\cd):\CR\to [\bH]$ is a Nevanlinna operator function called a characteristic matrix of the system \eqref{1.2} corresponding to $R(\l)$. By using the matrix $\Om(\cd)$ one constructs a spectral function generating an eigenfunction expansion of the system \eqref{1.2} (see e.g. \cite{DLS93}).

A somewhat other approach in the theory of generalized resolvents of $\Tmi$ is based on an application of boundary problems for the system \eqref{1.2}. Namely, assume that \eqref{1.2} is a Hamiltonian system and that $\Tmi $ has minimal  deficiency indices $n_\pm(\Tmi)=\dim H$. Then for each $\l\in\CR$ there exists a unique operator solution $v(t,\l)(\in [H, H\oplus H])$ of Eq. \eqref{1.5} such that $v(\cd,\l)h\in\LI, \; h\in H,$ and
\begin {equation}\label{1.7}
v(a,\l)=\begin{pmatrix}m(\l)\cr -I_H \end{pmatrix}:H\to H\oplus H, \quad \l\in\CR.
\end{equation}
Equality \eqref{1.7} defines a Nevanlinna operator function $m(\cd):\CR\to [H]$ called the Titchmarsh-Weyl coefficient (see e.g. \cite{HinSch93}). Moreover, the following holds: (1) for each generalized resolvent $R(\l)$ of $\Tmi$ there exists a unique holomorphic operator function $C_a(\cd):\CR\to  [H\oplus H,H] $ satisfying
\begin {equation}\label{1.8}
\ran C_a(\l)=H, \quad i\im \l\cd C_a(\l)J C_a^*(\l)\geq 0, \quad C_a(\l)J C_a^*(\ov\l)=0,\quad \l\in\CR
\end{equation}
and such that a function $y(t)=(R(\l)f)(t), \; f=f(\cd)\in \LI,$ is an $L_\D^2$-solution of the following boundary problem with $\l$-depending boundary condition:
\begin {gather}
J y'-B(t)y = \l \D(t) y+\D(t)f(t), \quad t\in\cI\label{1.9}\\
C_a(\l)y(a)= 0, \quad \l\in\CR.\label{1.10}
\end{gather}

(2) The characteristic matrix $\Om(\cd)$ corresponding to $R(\l)$ is of the form

\begin {equation}\label{1.11}
\Om(\l)=\begin{pmatrix}
m(\l)-m(\l)(\tau(\l) +m(\l))^{-1}m(\l) & -\tfrac 1 2 I +m(\l) (\tau\l) +m(\l))^{-1} \cr -\tfrac 1 2 I +(\tau(\l) +m(\l))^{-1}m(\l) & -(\tau(\l) + m(\l))^{-1}\end{pmatrix},
\end{equation}
where $\tau(\l):=\ker C_a(\l), \;\l\in\CR,$  is a Nevanlinna family of linear relations in $H$.

Statement (1) readily follows from the results of \cite{DijSno74,Sht70}, while statement (2) was proved in \cite{DLS88} (for the Sturm-Liouville operator see \cite{Sht55}).

Note that the case $n_+(\Tmi)=n_-(\Tmi)>\dim H$ is more complicated, because in this case only one boundary condition \eqref{1.10} at the endpoint $a$ is not sufficient for construction of a spectral function of the system \eqref{1.2}.

In the present paper we extend the above statements to general (not necessarily Hamiltonian) symmetric systems \eqref{1.2} with $n_-(\Tmi)\leq n_+(\Tmi)$. Our main result is a description of all generalized resolvents and characteristic matrices of such systems immediately in terms of boundary conditions. We describe all characteristic matrices by analogy with formula \eqref{1.11} and also by the formula similar to the well known Krein formula for resolvents.

To simplify the presentation of our results we assume within this section that  system \eqref{1.2} satisfies $n_+(\Tmi)= n_-(\Tmi)$. We show that in this case there exist a finite-dimensional Hilbert space  $\cH_b$ and a surjective linear mapping $\G_b:\dom\Tma\to \bH_b$ such that
\begin {equation*}
[y,z]_b(=\lim_{t \uparrow b}(J y(t),z(t)))=(J_b\G_b y,\G_b z), \quad y,z \in\dom\Tma.
\end{equation*}
Here $\bH_b=\cH_b\oplus \wh H\oplus \cH_b$ and $J_b$ is an operator in $\bH_b$ given by
\begin {equation}\label{1.12}
J_b=\begin{pmatrix} 0& 0& -I_{\cH_b}\cr 0 & i I_{\wh H} & 0 \cr I_{\cH_b} & 0 &
0\end{pmatrix}:\underbrace{\cH_b\oplus \wh H\oplus \cH_b }_{\bH_b}\to
\underbrace{\cH_b\oplus \wh H\oplus \cH_b}_{\bH_b}.
\end{equation}
In fact, $\G_b y $ is a singular boundary value of a function $y$  in the sense of \cite[Chapter 13.2]{DunSch} (for more details see Remark 3.5 in \cite{AlbMalMog13}).

Assume that $\cH_b$ and $\G_b$ are fixed and let $\cH=H\oplus\wh H\oplus\cH_b$. With each Nevanlinna family of linear relations (in particular operators) $\tau=\tau(\l)$ in $\cH$ we associate a pair of holomorphic  operator functions $C_a(\l)=C_{\tau,a}(\l)(\in [\bH,\cH])$ and  $C_b(\l)=C_{\tau,b}(\l)(\in [\bH_b,\cH])$ satisfying for all $\l\in\CR$ the relations (cf. \eqref{1.8})
 \begin {gather}
\ran (C_a(\l), \; C_b(\l))=\cH \label{1.14}\\
i \im \l\cd(C_a(\l)J C_a^*(\l)-C_b(\l)J_b C_b^*(\l))\geq 0, \quad C_a(\l)J
C_a^*(\ov\l)=C_b(\l)J_b C_b^*(\ov\l)\label{1.15}
\end{gather}
We show that for each generalized resolvent $R(\l)$ of $\Tmi$ there exists a unique Nevanlinna family of linear relations $\tau=\tau(\l)$ in $\cH$ such that a function $y(t)=(R(\l)f)(t), \; f=f(\cd)\in \LI,$ is an $L_\D^2$-solution of the following boundary problem (cf. \eqref{1.9},  \eqref{1.10})
\begin {gather}
J y'-B(t)y = \l \D(t) y+\D(t)f(t), \quad t\in\cI\label{1.16}\\
C_{\tau,a}(\l)y(a)+C_{\tau,b}(\l)\G_b y= 0, \quad \l\in\CR.\label{1.17}
\end{gather}
Note, that  \eqref{1.17} is a boundary condition imposed on boundary values of a function $y\in\dom\Tma$ (more precisely, on the regular value $y(a)$ and singular value $\G_b y$). One may consider
$\tau=\tau(\l)$ as a (Nevanlinna) boundary parameter, since $R(\l)$ runs over the set of all generalized resolvents of $\Tmi$ when $\tau$ runs over the set of all
Nevanlinna families of linear relations in $\cH$. To indicate this fact explicitly we write $R(\l)=R_\tau(\l)$ and $\Om(\l)=\Om_\tau(\l)$ for the generalized resolvent of $\Tmi$ and the corresponding characteristic matrix respectively.

The boundary  problem \eqref{1.16}, \eqref{1.17} defines a canonical
resolvent $R_\tau(\l)$ of $\Tmi$ if and only if $\tau=\tau^*$.  In this case  $C_{\tau,a}(\l)\equiv C_a, \; C_{\tau,b}(\l)\equiv C_b$ and the operators $C_a$ and $C_b$ satisfy
\begin {equation}\label{1.18}
\ran (C_a,\, C_b)=\cH \;\;\; \text{and} \;\;\; C_a J C_a^*=C_b J C_b^*.
\end{equation}
Moreover, $R_\tau(\l)=(\wt T^\tau - \l)^{-1}$ with   $\wt T^\tau = (\wt T^\tau)^*$  given by the  boundary conditions:
\begin {equation}\label{1.19}
\wt T^\tau =\{\{y,f\}\in\Tma :C_a y(a)+C_b\G_by=0\}.
\end{equation}
Thus, the equalities \eqref{1.19} and \eqref{1.18} gives a parametrization of all self-adjoint extensions $\wt T=\wt T^\tau$ of $\Tmi$ in terms of boundary conditions.  Note that in the case of the regular endpoint $b$  \eqref{1.18} and \eqref{1.19} take the form  of self-adjoint boundary conditions from \cite{Atk,GK}. Moreover, for Hamiltonian systems with singular endpoint $b$ the description of all extensions $\wt T=\wt T^*$ of $\Tmi$ in the form \eqref{1.18}, \eqref{1.19} was obtained in \cite{Kra89}.

It turns out that for each boundary parameter $\tau$ there exists a unique $[\bH]$-valued operator solution $Z_\tau(\cd,\l), \;\l\in\CR,$ of Eq. \eqref{1.5} such that $Z_\tau(\cd,\l)h\in\LI, \; h\in \bH,$ and the following boundary condition is satisfied:
\begin {equation*}
C_{\tau,a}(\l)(Z_\tau(a,\l)+J)h+ C_{\tau,b}(\l)\G_b (Z_\tau(\cd,\l)h)=0, \quad h\in\bH, \;\;\l\in\CR.
\end{equation*}
Moreover, the characteristic matrix $\Om_\tau(\cd)$ is
\begin {equation}\label{1.21}
\Om_\tau(\l)=Z_\tau(a,\l)+\tfrac 1 2 J, \quad \l\in\CR
\end{equation}
and the following inequality holds
\begin {equation} \label{1.22}
(\im \l)^{-1}\cd \im \Om_\tau(\l)\geq \int_\cI Z_\tau^*(t,\l) \D(t)
Z_\tau(t,\l) \, dt, \quad  \l\in\CR.
\end{equation}
Note that definition of the characteristic matrix $\Om_\tau(\cd)$ by means of \eqref{1.21} is similar to that of the Titchmarsh-Weyl  coefficient $m(\cd)$ by means of \eqref{1.7}. Observe also that formula \eqref{1.22} is similar to well-known formulas for various classes of boundary problems (see e.g. \cite{Gor66,Ber}).

The main result of the paper is a parametrization of all characteristic matrices $\Om(\cd)$ of the system \eqref{1.2} immediately in terms of the boundary parameter $\tau$. This result can be formulated in the form of the following theorem.
\begin{theorem}\label{th1.1}
There exist operator functions $\Om_0(\l)(\in [\bH]), \; S(\l)(\in [\cH,\bH])$ and a Nevanlinna operator function $M(\l)(\in [\cH]), \; \l\in\CR,$ such that the equality
\begin {equation}\label{1.23}
\Om(\l)=\Om_\tau(\l)=\Om_0(\l)-S(\l)(\tau(\l)+M(\l))^{-1}S^*(\ov\l), \quad \l\in\CR
\end{equation}
establishes a bijective correspondence between all (Nevanlinna) boundary parameters $\tau=\tau(\l)$ and all characteristic matrices $\Om(\cd)$ of the system \eqref{1.2}. Moreover, for each boundary parameter $\tau$  the corresponding characteristic matrix $\Om(\l)=\Om_\tau(\l)$ admits the representation
\begin {equation}\label{1.24}
\Om(\l)=X \wt \Om_\tau(\l)X^*, \quad \l\in\CR,
\end{equation}
where $X\in [\cH\oplus\cH,\bH]$ is a certain operator (see \eqref{4.98}) and $\wt \Om_\tau(\cd):\CR\to [\cH\oplus\cH]$ is a Nevanlinna operator function given by the block matrix representation (cf. \eqref{1.11})
\begin {equation}\label{1.25}
\wt\Om_{\tau}(\l)=\begin{pmatrix} M(\l)-M(\l)(\tau(\l) +M(\l))^{-1}M(\l) & -\tfrac 1 2 I_{\cH} +M(\l) (\tau\l) +M(\l))^{-1} \cr -\tfrac 1 2 I_{\cH} +(\tau(\l) +M(\l))^{-1}M(\l) & -(\tau(\l) + M(\l))^{-1} \end{pmatrix}
\end{equation}
\end{theorem}
Note that the operator functions $\Om_0(\cd), \; S(\cd)$ and $M(\cd)$ in \eqref{1.23} are defined in terms of the boundary values of respective $L_\D^2$-operator solutions of Eq.  \eqref{1.5}. Observe also  that in the case of the Hamiltonian system \eqref{1.2} with $n_\pm(\Tmi)=\dim H$ one has $\cH=H, \; X=I_{\bH}, \; M(\l)=m(\l)$ and hence $\Om(\l)(=\Om_\tau(\l))=\wt \Om_\tau(\l)$. This implies that equality \eqref{1.11} is a particular case of \eqref{1.24}, \eqref{1.25}.
\section{Preliminaries}
\subsection{Notations}
The following notations will be used throughout the paper: $\gH$, $\cH$ denote
Hilbert spaces; $[\cH_1,\cH_2]$  is the set of all bounded linear operators
defined on the Hilbert space $\cH_1$ with values in the Hilbert space $\cH_2$;
$[\cH]:=[\cH,\cH]$; $\bC_+\,(\bC_-)$ is the upper (lower) half-plane  of the
complex plane.

Let $\wt \cH$ be a Hilbert space and let $\cH$ be a subspace in $\wt \cH$. We
denote by $P_{\wt\cH,\cH}(\in [\wt\cH,\cH])$ the orthoprojection in $\wt\cH$
onto $\cH$. Moreover, we denote by $I_{\cH,\wt\cH}(\in [\cH,\wt\cH])$ the
embedding operator of the subspace $\cH$ into $\wt\cH$. It is clear that
$P_{\wt\cH,\cH}^*=I_{\cH,\wt\cH}$.

Recall that a closed linear   relation   from $\cH_0$ to $\cH_1$ is a closed
linear subspace in $\cH_0\oplus\cH_1$. The set of all closed linear relations
from $\cH_0$ to $\cH_1$ (in $\cH$) will be denoted by $\C (\cH_0,\cH_1)$
($\C(\cH)$). A closed linear operator $T$ from $\cH_0$ to $\cH_1$  is
identified  with its graph $\text {gr}\, T\in\CA$.

For a linear relation $T\in\C (\cH_0,\cH_1)$  we denote by $\dom T,\,\ran T,
\,\ker T$ and $\mul T$  the domain, range, kernel and the multivalued part of
$T$ respectively. Recall that $\mul T$ ia a linear manifold in $\cH_1$ defined by
\begin{gather*}
\mul T:=\{h_1\in \cH_1:\{0,h_1\}\in T\}.
\end{gather*}
Recall also that the inverse and adjoint linear relations of $T$ are the
relations $T^{-1}\in\C (\cH_1,\cH_0)$ and $T^*\in\C (\cH_1,\cH_0)$ defined by
\begin{gather}
T^{-1}=\{\{h_1,h_0\}\in\cH_1\oplus\cH_0:\{h_0,h_1\}\in T\}\nonumber\\
T^* = \{\{k_1,k_0\}\in \cH_1\oplus\cH_0:\, (k_0,h_0)-(k_1,h_1)=0, \;
\{h_0,h_1\}\in T\}\label{2.0}.
\end{gather}
For a linear relation $T\in \C(\cH)$ we denote by $\rho (T):=\{\l \in
\bC:(T-\l)^{-1}\in [\cH]\}$  the resolvent set of $T$.

Recall also that an operator function $\Phi
(\cd):\bC\setminus\bR\to [\cH]$ is called a Nevanlinna function if it is
holomorphic and satisfies $\im\, \l\cd \im \Phi (\l)\geq 0 $ and $\Phi ^*(\l)=\Phi (\ov \l), \; \l\in\bC\setminus\bR$.
\subsection{The classes $\wt R_+(\cH_0,\cH_1)$ and $\wt R(\cH)$}\label{sub2.2}
Let $\cH_0$ be a Hilbert space, let $\cH_1$ be a subspace in $\cH_0$ and let
$\tau =\{\tau_+,\tau_-\}$ be a collection of holomorphic functions
$\tau_\pm(\cd):\bC_\pm\to\CA$. In the paper we systematically deal with
collections $\pair$ of the special class $\wt R_+(\cH_0,\cH_1)$. Definition and
detailed characterization of this class can be found in our paper
\cite{Mog13.2} (see also \cite{Mog06.1,Mog11,AlbMalMog13}, where the notation
$\wt R (\cH_0,\cH_1)$ were used instead of $\wt R_+(\cH_0,\cH_1)$). If
$\dim\cH_1<\infty$, then according to \cite{Mog13.2} the collection $\pair\in
\wt R_+(\cH_0,\cH_1)$ admits the representation
\begin {equation}\label{2.1}
\tau_+(\l)=\{(C_0(\l),C_1(\l));\cH_0\}, \;\;\l\in\bC_+; \;\;\;\;
\tau_-(\l)=\{(D_0(\l),D_1(\l));\cH_1\}, \;\;\l\in\bC_-
\end{equation}
by means of two pairs of holomorphic operator functions
\begin {equation*}
(C_0(\l),C_1(\l)):\cH_0\oplus\cH_1\to\cH_0, \;\;\l\in\bC_+,\;\; \text{and}\;\;
(D_0(\l),D_1(\l)):\cH_0\oplus\cH_1\to\cH_1, \;\;\l\in\bC_-
\end{equation*}
(more precisely, by equivalence classes of such pairs). The equalities
\eqref{2.1} mean that
\begin {equation}\label{2.2}
\begin{array}{c}
\tau_+(\l)=\{\{h_0,h_1\}\in\cH_0\oplus\cH_1: C_0(\l)h_0+C_1(\l)h_1=0\},
\;\;\;\l\in\bC_+\\
\tau_-(\l)=\{\{h_0,h_1\}\in\cH_0\oplus\cH_1: D_0(\l)h_0+D_1(\l)h_1=0\},
\;\;\;\l\in\bC_-.
\end{array}
\end{equation}
In \cite{Mog13.2} the class $\wt R_+(\cH_0,\cH_1)$ is characterized both in
terms of $\CA$-valued functions $\tau_\pm(\cd)$ and in terms of operator
functions $C_j(\cd)$ and $D_j(\cd), \; j\in\{0,1\},$ from \eqref{2.1}.

If $\cH_1=\cH_0=:\cH$, then the class $\wt R(\cH):=\wt R_+ (\cH,\cH)$ coincides
with the well-known class of Nevanlinna $\C (\cH)$-valued functions $\tau
(\cd)$ (see, for instance, \cite{DM00}). In this case the collection
\eqref{2.1} turns into the Nevanlinna pair
\begin {equation}\label{2.3}
\tau(\l)=\{(C_0(\l),C_1(\l));\cH\}, \quad \l\in\CR,
\end{equation}
with $C_0(\l),C_1(\l)\in [\cH]$. Recall also that the subclass $\wt
R^0(\cH)\subset \wt R(\cH)$ is defined as the set of all $\tau(\cd)\in\wt
R(\cH)$ such that $\tau(\l)\equiv \t(=\t^*), \; \l\in\CR$. This implies that
$\tau(\cd)\in \wt R^0(\cH)$  if and only if
\begin {equation}\label{2.4}
\tau(\l)\equiv \{(C_0,C_1);\cH\},\quad \l\in\CR,
\end{equation}
with some operators $C_0,C_1\in [\cH]$ satisfying  $\im (C_1C_0^*)=0 $ and
$0\in\rho (C_0\pm i C_1)$ (for more details see e.g. \cite[Remark
2.5]{AlbMalMog13}).
\subsection{Boundary triplets and Weyl functions for symmetric relations}
Recall that a linear relation $A\in\C (\gH)$ is called symmetric (self-adjoint)
if $A\subset A^*$ (resp. $A=A^*$).

 Let $A$ be a closed  symmetric linear relation in the Hilbert space $\gH$,
let $\gN_\l(A)=\ker (A^*-\l)\; (\l\in\bC)$ be a defect subspace of $A$, let
$\wh\gN_\l(A)=\{\{f,\l f\}:\, f\in \gN_\l(A)\}$ and let $n_\pm (A):=\dim
\gN_\l(A)\leq\infty, \; \l\in\bC_\pm,$ be deficiency indices of $A$.

The following definitions are well known.
\begin{definition}\label{def2.1a}
The operator function $R(\cd):\CR\to [\gH]$  is called a generalized resolvent
of $A$ if there exist a Hilbert space $\wt \gH\supset\gH$ and a self-adjoint
relation $\wt A\in \C (\wt\gH)$ such that $A\subset \wt A$ and the following
equality holds
\begin {gather}\label{2.5}
R(\l) =P_\gH (\wt A- \l)^{-1}\up \gH, \quad \l \in \CR.
\end{gather}
The relation $\wt A$ in \eqref{2.5} is called an exit space extension of $A$.
\end{definition}
\begin{definition}\label{def2.1b}
The generalized resolvent \eqref{2.5} is called canonical if $\wt \gH=\gH$,
i.e., if $R(\l)=(\wt A-\l)^{-1}, \; \l\in\CR,$ is the resolvent of $\wt A=\wt
A^* \in \C (\gH), \; \wt A\supset A$.
\end{definition}
Clearly, canonical resolvents  exist if and only if $n_+(A)=n_-(A)$.

Next we recall  definitions of boundary triplets, the corresponding Weyl
functions, and $\gamma$-fields following  \cite{DM91, Mal92,
Mog06.2,Mog13.2}.

 Assume that $\cH_0$ is a Hilbert space,  $\cH_1$ is a subspace
in $\cH_0$ and   $\cH_2:=\cH_0\ominus\cH_1$, so that $\cH_0=\cH_1\oplus\cH_2$.
Denote by $P_j$ the orthoprojection   in $\cH_0$ onto $\cH_j,\; j\in\{1,2\} $.

\begin{definition}\label{def2.2}
 A collection $\Pi_+=\bta$, where
$\G_j: A^*\to \cH_j, \; j\in\{0,1\},$ are linear mappings, is called a boundary
triplet for $A^*$, if the mapping $\G :\wh f\to \{\G_0 \wh f, \G_1 \wh f\}, \wh
f\in A^*,$ from $A^*$ into $\cH_0\oplus\cH_1$ is surjective and the following
Green's identity holds
\begin {equation*}
(f',g)-(f,g')=(\G_1  \wh f,\G_0 \wh g)_{\cH_0}- (\G_0 \wh f,\G_1 \wh
g)_{\cH_0}+i (P_2\G_0 \wh f,P_2\G_0 \wh g)_{\cH_2}
\end{equation*}
 holds for all $\wh
f=\{f,f'\}, \; \wh g=\{g,g'\}\in A^*$.
\end{definition}
\begin{proposition}$\,$\cite{Mog06.2}\label{pr2.3}
Let  $\Pi_+=\bta$ be a boundary triplet for $A^*$. Then:
\begin{enumerate}\def\labelenumi{\rm (\arabic{enumi})}
\item
$\;\dim \cH_1=n_-(A)\leq n_+(A)=\dim \cH_0$.
\item
$\ker \G_0\cap\ker\G_1=A$ and $\G_j$ is a bounded operator from $A^*$ onto
$\cH_j$.
\item
The equality $A_0:=\ker \G_0=\{\wh f\in A^*:\G_0 \wh f=0\}$ defines a maximal
symmetric extension $A_0$ of $A$ such that $\bC_+\subset \rho (A_0) $.
\end{enumerate}
\end{proposition}
\begin{proposition}\label{pr2.4}$\,$ \cite{Mog06.2}
Let  $\Pi_+=\bta$ be a boundary triplet for $A^*$. Denote also by $\pi_1$ the
orthoprojection  in $\gH\oplus\gH$ onto $\gH\oplus \{0\}$. Then the operators
$\G_0\up \wh \gN_\l (A), \;\l\in\bC_+,$ and $P_1\G_0\up \wh \gN_z (A),\;
z\in\bC_-,$ isomorphically map $\wh\gN_\l (A)$ onto $\cH_0$ and $\wh\gN_z(A)$
onto $\cH_1$ respectively. Therefore the equalities
\begin {equation}\label{2.7}
\begin{array}{c}
\g_{+} (\l)=\pi_1(\G_0\up\wh \gN_\l (A))^{-1}, \;\;\l\in\Bbb C_+,\\
 \g_{-} (z)=\pi_1(P_1\G_0\up\wh\gN_z (A))^{-1}, \;\; z\in\Bbb C_-,
\end{array}
\end{equation}
\begin{gather}
M_{+}(\l)h_0=\G_1\{\g_+(\l)h_0, \l\g_+(\l)h_0\}, \quad h_0\in\cH_0, \quad
\l\in\bC_+\label{2.8}
\end{gather}
correctly define the operator functions $\g_{+}(\cdot):\Bbb C_+\to[\cH_0,\gH],
\; \; \g_{-}(\cdot):\Bbb C_-\to[\cH_1,\gH]$ and $M_{+}(\cdot):\bC_+\to
[\cH_0,\cH_1]$, which are holomorphic on their domains. Moreover,
\begin {equation}\label{2.9}
M_+(\mu)-M_+^*(\l)P_1+iP_2=(\mu-\ov\l)\g_+^*(\l)\g_+(\mu), \qquad
\mu,\l\in\bC_+.
\end{equation}
where $M_+(\mu)$ is considered as an operator in $\cH_0$ (it  is possible in
view of the inclusion $\cH_1\subset\cH_0$).
\end{proposition}
It follows from \eqref{2.7} that
\begin {equation}\label{2.10}
\G_0\{\g_+(\l)h_0,\l \g_+(\l)h_0\}=h_0,  \qquad h_0\in\cH_0.
\end{equation}

\begin{definition}\label{def2.10}$\,$\cite{Mog06.2}
The operator functions $\g_\pm(\cd)$ and $M_+(\cd)$ defined in Proposition
\ref{pr2.4}  are called the $\g$-fields and the Weyl function, respectively,
corresponding to the boundary triplet $\Pi_+$.
\end{definition}
\begin{remark}\label{rem2.11}
 (1) If $\cH_0=\cH_1:=\cH$, then   the boundary triplet in the sense
of Definition \ref{def2.2} turns into the boundary triplet
$\Pi=\{\cH,\G_0,\G_1\}$ for $A^*$ in the sense of \cite{GorGor,Mal92}. In this
case $n_+(A)=n_-(A)(=\dim \cH)$, $A_0(=\ker \G_0)$ is a self-adjoint extension
of $A$ and the functions  $\g_\pm(\cd)$ and  $M_+(\cd)$ turn into the
$\g$-field $\g(\cd):\rho (A_0)\to [\cH,\gH]$ and  Weyl function $M(\cd):\rho
(A_0)\to [\cH]$ respectively introduced in \cite{DM91,Mal92}. Moreover, in this case $M(\cd)$ is a Nevanlinna operator function.

To avoid
misleading with using other definitions,  a boundary triplet $\Pi=\bt$ in the
sense of \cite{GorGor,Mal92} will be  called an \emph{ordinary boundary
triplet} for $A^*$.

(2) Along with $\Pi_+$ we define in \cite{Mog13.2}  a boundary triplet
$\Pi_-=\bta$ for $A^*$. Such a triplet is applicable to symmetric relations $A$
with $n_+(A)\leq n_-(A)$.
\end{remark}

\section{Decomposing boundary triplets for symmetric systems}
\subsection{Notations}
Let $\cI=[ a,b\rangle\; (-\infty < a< b\leq\infty)$ be an interval of the real
line (the symbol $\rangle$ means that the endpoint $b<\infty$  might be either
included  to $\cI$ or not). For a given finite-dimensional Hilbert space $\bH$
denote by $\AC$ the set of functions $f(\cd):\cI\to \bH$ which are absolutely
continuous on each segment $[a,\b]\subset \cI$ and let $AC (\cI):=AC(\cI;\bC)$.

Next assume that $\D(\cd)$ is an $[\bH]$-valued Borel measurable  functions on
$\cI$ integrable on each compact interval $[a,\b]\subset \cI$ and such that
$\D(t)\geq 0$. Denote  by $\lI$  the semi-Hilbert  space of  Borel measurable
functions $f(\cd): \cI\to \bH$ satisfying $||f||_\D^2:=\int\limits_{\cI}(\D
(t)f(t),f(t))_\bH \,dt<\infty$ (see e.g. \cite[Chapter 13.5]{DunSch}).  The
semi-definite inner product $(\cd,\cd)_\D$ in $\lI$ is defined by $
(f,g)_\D=\int\limits_{\cI}(\D (t)f(t),g(t))_\bH \,dt,\; f,g\in \lI$. Moreover,
let $\LI$ be the Hilbert space of the equivalence classes in $\lI$ with respect
to the semi-norm $||\cd||_\D$ and let $\pi$ be the quotient map from $\lI$ onto
$\LI$.

For a given finite-dimensional Hilbert space $\cK$ denote by $\lo{\cK}$ the set
of all Borel measurable  operator-functions $F(\cd): \cI\to [\cK,\bH]$ such
that $F(t)h\in \lI$ for each $h\in\cK$ (this condition is equivalent to
$\int\limits_{\cI}||\D^{\frac 1 2} (t)F(t)||^2\,dt <\infty$). Moreover, we let
$\cL_\D^2[\bH]:=\cL_\D^2[\bH,\bH]$.
\subsection{Symmetric systems}
In this subsection we provide some known results on symmetric systems of
differential equations following  \cite{GK, Kac03, LesMal03, Orc}.

Let $H$ and $\wh H$ be  finite-dimensional Hilbert spaces and let
\begin {equation}\label{3.0.1}
H_0=H\oplus\wh H, \quad \bH=H_0\oplus  H=H\oplus\wh H \oplus H.
\end{equation}
In the following we denote by $\cP_0,\; \wh\cP$ and  $\cP_1$ the orthoprojections
in $\bH$ onto the first, second and third component in the decomposition
$\bH=H\oplus\wh H \oplus H$ respectively.

Let as above $\cI=[ a,b\rangle\; (-\infty < a< b\leq\infty)$ be an interval in
$\bR$ . Moreover, let $B(\cd)$ and $\D(\cd)$ be $[\bH]$-valued Borel measurable
functions on $\cI$ integrable on each compact interval $[a,\b]\subset \cI$ and
satisfying $B(t)=B^*(t)$ and $\D(t)\geq 0$ a.e. on $\cI$ and let $J\in [\bH]$
be  operator \eqref{1.3}.

A first-order symmetric  system on an interval $\cI$ (with the regular endpoint
$a$) is a system of differential equations of the form
\begin {equation}\label{3.1}
J y'-B(t)y=\D(t) f(t), \quad t\in\cI,
\end{equation}
where $f(\cd)\in \lI$. Together   with \eqref{3.1} we consider also the
homogeneous system
\begin {equation}\label{3.2}
J y'(t)-B(t)y(t)=\l \D(t) y(t), \quad t\in\cI, \quad \l\in\bC.
\end{equation}
A function $y\in\AC$ is a solution of \eqref{3.1} (resp. \eqref{3.2}) if
equality \eqref{3.1} (resp. \eqref{3.2} holds a.e. on $\cI$. A function
$Y(\cd,\l):\cI\to [\cK,\bH]$ is an operator solution of  equation \eqref{3.2}
if $y(t)=Y(t,\l)h$ is a (vector) solution of this equation for every $h\in\cK$
(here $\cK$ is a Hilbert space with $\dim\cK<\infty$).

As is known there exists a unique $[\bH]$-valued operator solution
$Y_0(\cd,\l)$ of Eq. \eqref{3.2} satisfying $Y_0(a,\l)=I_{\bH}$. Moreover, each
operator solution $Y(\cd,\l)$ of Eq. \eqref{3.2} admits the representation
\begin {equation}\label{3.3}
Y(t,\l)=Y_0(t,\l)Y(a,\l), \qquad t\in\cI.
\end{equation}

In what follows  we always assume  that  system \eqref{3.1} is definite in the
sense of the following definition.
 \begin{definition}\label{def3.1}$\,$\cite{GK}
Symmetric system \eqref{3.1} is called definite if for each $\l\in\bC$ and each
solution $y$ of \eqref{3.2} the equality $\D(t)y(t)=0$ (a.e. on $\cI$) implies
$y(t)=0, \; t\in\cI$.
\end{definition}

As it is known \cite{Orc, Kac03, LesMal03} definite system \eqref{3.1} gives
rise to the \emph{maximal linear relations} $\tma$ and $\Tma$  in  $\lI$ and
$\LI$, respectively. They are given by
\begin {equation}\label{3.4}
\begin{array}{c}
\tma=\{\{y,f\}\in(\lI)^2 :y\in\AC \;\;\text{and}\;\; \qquad\qquad\qquad\qquad \\
\qquad\qquad\qquad\qquad\qquad\quad  J y'(t)-B(t)y(t)=\D(t) f(t)\;\;\text{a.e.
on}\;\; \cI \}
\end{array}
\end{equation}
and $\Tma=\{\{\pi y,\pi f\}:\{y,f\}\in\tma\}$. Moreover the Lagrange's identity
\begin {equation}\label{3.6}
(f,z)_\D-(y,g)_\D=[y,z]_b - (J y(a),z(a)),\quad \{y,f\}, \; \{z,g\} \in\tma.
\end{equation}
holds with
\begin {equation}\label{3.7}
[y,z]_b:=\lim_{t \uparrow b}(J y(t),z(t)), \quad y,z \in\dom\tma.
\end{equation}
Formula \eqref{3.7} defines the skew-Hermitian  bilinear form $[\cd,\cd]_b $ on
$\dom \tma$. By using this form one defines the \emph{minimal relations} $\tmi$
in $\lI$ and $\Tmi$ in $\LI$ via
\begin {equation*}
\tmi=\{\{y,f\}\in\tma: y(a)=0 \;\; \text{and}\;\; [y,z]_b=0\;\;\text{for
each}\;\; z\in \dom \tma \}.
\end{equation*}
and $\Tmi=\{\{\pi y,\pi f\}:\{y,f\}\in\tmi\} $. According to \cite{Orc,Kac03,
LesMal03} $\Tmi$ is a closed symmetric linear relation in $\LI$ and
$\Tmi^*=\Tma$.
\begin{remark}\label{rem3.1a}
It  is known (see e.g. \cite{LesMal03}) that the maximal relation $\Tma$
induced by the definite symmetric system \eqref{3.1} possesses the following
 property: for any $\{\wt y, \wt f \}\in \Tma $ there exists
a unique function $y\in \AC \cap \lI $ such that $y\in \wt y$ and $\{y,f\}\in
\tma$ for any $f\in\wt f$. Below we associate such a function $ y\in \AC \cap
\lI$ with each pair $\{\wt y, \wt f\}\in\Tma$.
\end{remark}

Denote by $\cN_\l,\; \l\in\bC,$ the linear space of solutions of the
homogeneous system \eqref{3.2} belonging to $\lI$. Definition \eqref{3.4} of
$\tma$ implies
\begin{equation*}
\cN_\l=\ker (\tma-\l)=\{y\in\lI:\; \{y,\l y\}\in\tma\}, \quad\l\in\bC,
\end{equation*}
and hence $\cN_\l\subset \dom\tma$. As usual, denote by $ n_\pm (\Tmi ):=\dim
\gN_\l (\Tmi), \quad \l\in\bC_\pm,$ the  deficiency indices of  $\Tmi$.  Since
the system \eqref{3.1} is definite, $ \pi\cN_\l=\gN_\l (\Tmi)$ and
$\ker(\pi\up\cN_\l)=\{0\},\;\; \l\in\bC$.  This implies that $n_\pm (\Tmi)=\dim
\cN_\l\leq n, \; \l\in\bC_\pm$.

The following lemma will be useful in the sequel.
\begin{lemma}\label{lem3.1b}$\,$\cite{AlbMalMog13}
For each operator solution $Y(\cd,\l)\in \lo{\cK}$ of Eq. \eqref{3.2} the
relation
\begin {equation} \label{3.10}
\cK\ni h\to (Y(\l) h)(t)=Y(t,\l)h \in\cN_\l.
\end {equation}
defines the linear mapping $Y(\l):\cK\to \cN_\l$. Moreover, if $F(\l):=\pi
Y(\l)(\in [\cK, \LI])$, then
\begin {equation}\label{3.11}
F^*(\l)\wt f=\int_\cI Y^*(t,\l)\D(t)f(t)\, dt,\qquad \wt f\in \LI,\quad f\in\wt
f.
\end{equation}
\end{lemma}
\subsection{Decomposing boundary triplets}\label{sub3.3}
We start this subsection with the following lemma.
\begin{lemma}$\,$\cite{AlbMalMog13} \label{lem3.2}
If $n_-(\Tmi)\leq n_+(\Tmi)$, then there exist a finite dimensional  Hilbert
space $\wt \cH_b$, a subspace $\cH_b\subset \wt \cH_b$  and a surjective linear
mapping
\begin{gather}\label{3.28}
\G_b=\begin{pmatrix}\G_{0b}\cr  \wh\G_b \cr  \G_{1b}\end{pmatrix} :\dom\tma\to
\wt \cH_b\oplus\wh H \oplus \cH_b
\end{gather}
such that for all $y,z \in \dom\tma$ the following identity is valid
\begin {equation} \label{3.29}
 [y,z]_b=(\G_{0b}y,\G_{1b}z)_{\wt\cH_b}-(\G_{1b}y,\G_{0b}z)_{\wt\cH_b}
+i(P_{\cH_b^\perp}\G_{0b}y, P_{\cH_b^\perp}\G_{0b}z)_{\wt\cH_b}+ i (\wh\G_b y,
\wh\G_b z)_{\wh H}
\end{equation}
(here $\cH_b^\perp=\wt\cH_b\ominus \cH_b$). Moreover, in the case $n_+(\Tmi)= n_-(\Tmi)$ (and only in this case) one has $\wt \cH_b=\cH_b$ and the identity \eqref{3.29} takes the form
\begin {equation*}
[y,z]_b=(\G_{0b}y,\G_{1b}z)_{\cH_b}-(\G_{1b}y,\G_{0b}z)_{\cH_b}+  i (\wh\G_b y,
\wh\G_b z)_{\wh H}
\end{equation*}
\end{lemma}
The following proposition is immediate from \cite[Proposition
3.6]{AlbMalMog13}.
\begin{proposition}\label{pr3.3}
Assume that $n_-(\Tmi)\leq n_+(\Tmi)$. Moreover, let
\begin {equation}\label{3.32}
y(t)=\{y_0(t),\,\wh y(t), \, y_1(t) \}(\in \underbrace{H\oplus\wh H \oplus
H}_{\bH}), \quad t\in\cI,
\end{equation}
be the representation of a function $y\in\dom\tma$ in accordance with the
decomposition \eqref{3.0.1} of $\bH$ and let $\G_b$ be the surjective linear
mapping \eqref{3.28} satisfying the identity \eqref{3.29}. Assume also that
$\cH_0$ and $\cH_1(\subset \cH_0)$ are finite dimensional Hilbert spaces
defined by
\begin {equation}\label{3.33}
\cH_0=\underbrace{H\oplus\wh H}_{H_0}\oplus\wt\cH_b=H_0 \oplus\wt\cH_b, \qquad
\cH_1=\underbrace{H\oplus\wh H}_{H_0}\oplus\cH_b=H_0\oplus \cH_b
\end{equation}
 and  $\G_j:\Tma\to\cH_j, \; j\in
\{0,1\},$ are  the operators given by
\begin {gather}
\G_0\{\wt y, \wt f\}=\{ - y_1(a),\;i(\wh y(a)-\wh\G_b y),\;
\G_{0b}y\}(\in  H\oplus\wh H\oplus\wt\cH_b),\qquad\qquad\qquad\quad\label{3.34}\\
\G_1\{\wt y, \wt f\}=\{y_0(a),\; \tfrac 1 2(\wh y(a)+\wh\G_b)y, \;
-\G_{1b}y\}(\in H\oplus\wh H\oplus \cH_b), \quad \{\wt y, \wt
f\}\in\Tma.\label{3.35}
\end{gather}
(here $y_0(a), \wh y(a)$ and $y_1(a)$ are taken from the representation
\eqref{3.32} of a function  $y\in\dom\tma$, which  corresponds to $\{\wt y, \wt
f\}\in\Tma$ according to Remark \ref{rem3.1a}). Then the collection
$\Pi_+=\bta$ is a boundary triplet for $\Tma$ and the (maximal symmetric)
relation $A_0(=\ker \G_0)$ is
\begin {equation}\label{3.36}
A_0=\{\{\wt y,\wt f\}\in\Tma: y_1(a)=0,\; \wh y(a)=\wh \G_b y, \; \G_{0b}y=0\}.
\end{equation}

If in addition $n_+(\Tmi)=n_-(\Tmi)$, then $\Pi_+$ turns into an ordinary
boundary triplet $\Pi=\bt$ for $\Tma$, with $\cH=H_0\oplus\cH_b$  and the
mappings  $\G_0, \G_1: \Tma\to\cH,$  given by \eqref{3.34} and \eqref{3.35}
with $\wt\cH_b=\cH_b$. Moreover, in this case $A_0=A_0^*$.
\end{proposition}
\begin{definition}\label{def3.4}
The boundary triplet $\Pi_+=\bta$  constructed in Proposition \ref{pr3.3} is
called a decomposing  boundary triplet for $\Tma$.
\end{definition}
In the sequel we suppose (unless otherwise stated) that the following
assumptions are fulfilled:
\begin{itemize}
\item[(A1)]
The  system \eqref{3.1} satisfies $n_-(\Tmi)\leq n_+(\Tmi)$:
\item[(A2)]
 $\wt\cH_b$ and $\cH_b(\subset \wt\cH_b)$ are finite dimensional
Hilbert spaces and $\G_b$ is the surjective linear mapping \eqref{3.28} such
that \eqref{3.29} holds.
\end{itemize}

The following two propositions directly follows from \cite[Propositions 4.4 and
4.5]{AlbMalMog13}.
\begin{proposition}\label{pr3.5}
{\rm (1)}  For every $\l\in\CR$ there exists a unique  operator solution
$v_0(\cd,\l)\in\lo{H_0}$ of Eq. \eqref{3.2} such that
\begin{gather}
\cP_1 v_0(a,\l)=-P_{H_0,H}, \qquad i (\wh \cP v_0(a,\l)-\wh\G_b
v_0(\l))=P_{H_0,\wh H},
\;\;\;\; \l\in\CR \label{3.37}\\
\G_{0b} v_0(\l)=0, \;\;\;\;\l\in\bC_+;\qquad P_{\cH_b}\G_{0b} v_0(\l)=0,
\;\;\;\;\l\in\bC_-.\nonumber
\end{gather}

{\rm (2)} For every $\l\in\bC_+ \; (\l\in\bC_-)$ there exists a unique operator
solution $u_+(\cd,\l)\in\lo{\wt\cH_b}$ (resp. $u_-(\cd,\l)\in\lo{\cH_b}$) of
Eq. \eqref{3.2} such that
\begin{gather}
\cP_1 u_\pm(a,\l)=0, \qquad i(\wh\cP u_\pm(a,\l)-\wh\G_bu_\pm(\l))=0,\quad
\l\in\bC_\pm,\label{3.38}\\
\G_{0b}u_+(\l)=I_{\wt\cH_b},\;\; \l\in\bC_+; \qquad
P_{\cH_b}\G_{0b}u_-(\l)=I_{\cH_b},\;\; \l\in\bC_-.\nonumber
\end{gather}
Here $v_0(\l)$ and $u_\pm(\l)$ denote  linear mappings from Lemma \ref{lem3.1b}
for the solutions $v_0(\cd,\l)$ and $u_{\pm}(\cd,\l)$, respectively.
\end{proposition}
\begin{proposition}\label{pr3.6}
 Let $v_0(\cd,\l)$ and $u_\pm(\cd,\l)$ be
the operator solutions from Proposition \ref{pr3.5}, let $Z_+(\cd,\l)\in
\lo{\cH_0}$ and $Z_-(\cd,\l)\in \lo{\cH_1}$ be the operator solutions of Eq.
\eqref{3.2} given by
\begin{gather}
Z_+(t,\l)=(v_0(t,\l),\, u_+(t,\l)):H_0\oplus\wt\cH_b\to \bH, \quad
\l\in\bC_+\label{3.40}\\
Z_-(t,\l)=(v_0(t,\l),\, u_-(t,\l)):H_0\oplus\cH_b\to \bH, \quad
\l\in\bC_-\label{3.41}.
\end{gather}
and let $Z_\pm(\l)$ be the linear mappings from Lemma \ref{lem3.1b} for the
solutions $Z_\pm(\cd,\l)$. Moreover, let $\Pi_+=\bta$ be a decomposing boundary
triplet \eqref{3.33}--\eqref{3.35} for $\Tma$. Then the $\g$-fields
$\g_\pm(\cd)$ of the triplet $\Pi_+$ are
\begin {equation}\label{3.42}
\g_+(\l)=\pi Z_+(\l),\quad \l\in\bC_+;\qquad \g_-(\l)=\pi Z_-(\l),\quad
\l\in\bC_-
\end{equation}
and the corresponding Weyl function $M_+(\cd)$ admits the block-matrix
representation
\begin{gather}\label{3.43}
M_+(\l)=\begin{pmatrix} m_0(\l )& M_{2+}(\l) \cr M_{3+}(\l) & M_{4+}(\l)
\end{pmatrix}: \underbrace{H_0\oplus\wt\cH_b}_{\cH_0}\to \underbrace{H_0
\oplus\cH_b}_{\cH_1},\;\;\;\l\in\bC_+
\end{gather}
 with the entries defined by
\begin{gather}
m_0(\l)=(\cP_0 +\wh \cP)v_0(a,\l)+\tfrac i 2 P_{\wh H}, \;\;\;\;\; M_{2+}(\l)
=(\cP_0+\wh\cP)u_+(a,\l)\label{3.45}\\
M_{3+}(\l)=-\G_{1b}v_0(\l), \qquad M_{4+}(\l)=-\G_{1b}u_+(\l).\label{3.47}
\end{gather}
Moreover, for each $\l\in\bC_-$ the following equalities hold
\begin {equation}\label{3.49}
m_0^*(\ov\l)=(\cP_0 +\wh\cP)v_0(a,\l)+\tfrac i 2 P_{\wh H}, \quad
M_{3+}^*(\ov\l)=(\cP_0 +\wh\cP)u_-(a,\l).
\end{equation}
In \eqref{3.45} and \eqref{3.49} $P_{\wh H}\in [H_0]$ is the orthoprojection in
$H_0$ onto $\wh H$.
\end{proposition}
\begin{corollary}\label{cor3.7}$\,$\cite{AlbMalMog13}
Let $n_+(\Tmi)=n_-(\Tmi)$ and let $A_0=A_0^*$ be the  extension \eqref{3.36} of
$\Tmi$. Then for every $\l\in\rho (A_0)$ there exists a unique pair of
operator-valued solutions $v_0(\cd,\l)\in\lo{H_0}$ and $u(\cd,\l)\in\lo{\cH_b}$
of Eq. \eqref{3.2} satisfying
\begin{gather*}
\cP_1 v_0(a,\l)=-P_{H_0,H}, \quad i(\wh\cP v_0(a,\l)-\wh\G_b
v_0(\l))=P_{H_0,\wh H},\quad \G_{0b}
v_0(\l)=0, \;\; \l\in\rho (A_0),  \\
\cP_1 u(a,\l)=0, \quad i(\wh\cP u(a,\l)-\wh\G_b u(\l))=0, \quad
\G_{0b}u(\l)=I_{\cH_b},\;\;\l\in\rho (A_0).
\end{gather*}
Assume,  in addition, that   $\Pi=\bt$ is an ordinary decomposing boundary
triplet \eqref{3.33}--\eqref{3.35} for $\Tma$.   Then the corresponding Weyl
function $M(\cd)$ is
\begin {gather}
M(\l)=\begin{pmatrix} m_0(\l )& M_{2}(\l) \cr M_{3}(\l) & M_{4}(\l)
\end{pmatrix}: H_0\oplus \cH_b\to H_0\oplus\cH_b,
\;\;\;\l\in\rho (A_0)\label{3.50}\\
m_0(\l)=(\cP_0+\wh\cP)v_0(a,\l)+\tfrac i 2 P_{\wh H},\qquad
M_{2}(\l)=(\cP_0+\wh\cP)u(a,\l),\label{3.51}\\
M_{3}(\l)=-\G_{1b}v_0(\l), \qquad M_{4}(\l)=-\G_{1b}u(\l), \quad \l\in\rho
(A_0).\label{3.52}
\end{gather}
\end{corollary}
\section {Generalized resolvents and characteristic matrices of symmetric systems}
\subsection{Generalized resolvents}
\begin{definition}\label{def4.1}
Let $\cH_0$ and $\cH_1$ be finite dimensional Hilbert spaces \eqref{3.33}. Then
a boundary parameter $\tau$ is a collection $\pair\in\RH$ of the form
\eqref{2.1}.
\end{definition}
In the case of equal deficiency indices $n_+(\Tmi)=n_-(\Tmi)$ one has $\wt
\cH_b=\cH_b, \; \cH_0=\cH_1=:\cH$ and a boundary parameter is an operator pair
$\tau\in\wt R(\cH)$ defined by \eqref{2.3}. If in addition $\tau\in\wt
R^0(\cH)$, then a boundary parameter will be called self-adjoint. Such a
boundary parameter $\tau$ admits the representation as a self-adjoint operator
pair \eqref{2.4}.

For each boundary parameter $\pair$ of the form \eqref{2.1} we assume that
\begin {gather}
C_0(\l)=(C_{0a}(\l),\; \wh C_0(\l), \; C_{0b}(\l)):H\oplus\wh H\oplus \wt
\cH_b\to \cH_0, \quad \l\in\bC_+\label{4.1}\\
C_1(\l)=(C_{1a}(\l),\; \wh C_1(\l), \; C_{1b}(\l)):H\oplus\wh H\oplus
\cH_b\to \cH_0, \quad \l\in\bC_+\label{4.2}\\
D_0(\l)=(D_{0a}(\l),\; \wh D_0(\l), \; D_{0b}(\l)):H\oplus\wh H\oplus \wt
\cH_b\to \cH_1, \quad \l\in\bC_-\label{4.3}\\
D_1(\l)=(D_{1a}(\l),\; \wh D_1(\l), \; D_{1b}(\l)):H\oplus\wh H\oplus \cH_b\to
\cH_1, \quad \l\in\bC_-\label{4.4}
\end{gather}
are the block matrix representations of $C_j(\l)$ and $D_j(\l),\; j\in\{0,1\}$.

If $n_+(\Tmi)=n_-(\Tmi)$, then for each boundary parameter \eqref{2.3} we
assume that
\begin {gather}
C_0(\l)=(C_{0a}(\l),\; \wh C_0(\l), \; C_{0b}(\l)):H\oplus\wh H\oplus
\cH_b\to \cH, \quad \l\in\CR\label{4.5}\\
C_1(\l)=(C_{1a}(\l),\; \wh C_1(\l), \; C_{1b}(\l)):H\oplus\wh H\oplus \cH_b\to
\cH, \quad \l\in\CR\label{4.6}
\end{gather}
are the block matrix representations of $C_0(\l)$ and $C_1(\l)$.

In the case of a self-adjoint boundary parameter \eqref{2.4} the equalities
\eqref{4.5} and \eqref{4.6} take the form
\begin {equation}\label{4.7}
C_0=(C_{0a}, \; \wh C_0, \; C_{0b}):H\oplus\wh H\oplus \cH_b\to \cH, \quad
C_1=(C_{1a}, \; \wh C_1, \; C_{1b}):H\oplus\wh H\oplus \cH_b\to \cH
\end{equation}
\begin {lemma}\label{lem4.2}
Let $\wt\cH_b$ be decomposed as $\wt\cH_b=\cH_b\oplus \cH_b^\perp$ and let
\begin {equation*}
\bH_b:=\wt\cH_b\oplus \wh H\oplus \cH_b=\cH_b\oplus(\cH_b^\perp\oplus \wh
H)\oplus \cH_b,
\end{equation*}
\begin {equation}\label{4.8}
J_b=\begin{pmatrix} 0& 0& -I_{\cH_b}\cr 0 & i I_{\cH_b^\perp\oplus \wh H} & 0
\cr I_{\cH_b} & 0 & 0\end{pmatrix}:\underbrace{\cH_b\oplus(\cH_b^\perp\oplus
\wh H)\oplus \cH_b }_{\bH_b}\to \underbrace{\cH_b\oplus(\cH_b^\perp\oplus \wh
H)\oplus \cH_b}_{\bH_b}.
\end{equation}
Then the equalities
\begin {gather}
C_a(\l)=(-C_{1a}(\l),\; i\wh C_0(\l)-\tfrac 1 2 \wh C_1(\l), \; -
C_{0a}(\l)):H\oplus\wh H\oplus H\to \cH_0, \quad \l\in\bC_+\label{4.10}\\
C_b(\l)=(C_{0b}(\l),\; -i\wh C_0(\l)-\tfrac 1 2 \wh C_1(\l), \;
C_{1b}(\l)):\wt\cH_b\oplus\wh H\oplus \cH_b\to \cH_0, \quad \l\in\bC_+
\label{4.11}\\
D_a(\l)=(-D_{1a}(\l),\; i\wh D_0(\l)-\tfrac 1 2 \wh D_1(\l), \; -
D_{0a}(\l)):H\oplus\wh H\oplus H\to \cH_1, \quad \l\in\bC_-\label{4.12}\\
D_b(\l)=(D_{0b}(\l),\; -i\wh D_0(\l)-\tfrac 1 2 \wh D_1(\l), \;
D_{1b}(\l)):\wt\cH_b\oplus\wh H\oplus \cH_b \to \cH_1, \quad
\l\in\bC_-\label{4.13}
\end{gather}
establish a bijective correspondence between all boundary parameters $\pair$ of
the form \eqref{2.1} and \eqref{4.1}--\eqref{4.4} and all collections of
holomorphic operator functions
\begin {equation}\label{4.14}
(C_a(\l), \; C_b(\l)):\bH\oplus\bH_b\to \cH_0, \;\;\l\in\bC_+;\quad (D_a(\l),
\; D_b(\l)):\bH\oplus\bH_b\to \cH_1, \;\;\l\in\bC_-
\end{equation}
satisfying
\begin {gather}
\ran (C_a(\l), \; C_b(\l))=\cH_0, \;\;\l\in\bC_+;\qquad \ran (D_a(\l), \;
D_b(\l))= \cH_1, \;\;\l\in\bC_-\label{4.15}\\
i (C_a(\l)J C_a^*(\l)-C_b(\l)J_b C_b^*(\l))\geq 0, \quad i (D_a(\l)J
D_a^*(\l)-D_b(\l)J_b D_b^*(\l))\leq 0\label{4.16}\\
C_a(\l)J D_a^*(\ov\l)=C_b(\l)J_b D_b^*(\ov\l), \;\; \l\in\bC_+\label{4.17}
\end{gather}
If in addition $n_+(\Tmi)=n_-(\Tmi)$, then $\cH_b^\perp=\{0\}, \;
\bH_b=\cH_b\oplus \wh H\oplus \cH_b$, $J_b$ takes the form \eqref{1.12}
and the equalities
\begin {gather}
C_a(\l)=(-C_{1a}(\l),\; i\wh C_0(\l)-\tfrac 1 2 \wh C_1(\l), \; -
C_{0a}(\l)):H\oplus\wh H\oplus H\to \cH_, \quad \l\in\CR\label{4.18}\\
C_b(\l)=(C_{0b}(\l),\; -i\wh C_0(\l)-\tfrac 1 2 \wh C_1(\l), \;
C_{1b}(\l)):\cH_b\oplus\wh H\oplus \cH_b\to \cH, \quad \l\in\CR\label{4.19}
\end{gather}
establish a bijective correspondence between all boundary parameters $\tau$ of
the form \eqref{2.3} and \eqref{4.5}, \eqref{4.6} and all holomorphic operator
functions
\begin {equation}\label{4.20}
(C_a(\l), \; C_b(\l)):\bH\oplus\bH_b\to \cH, \;\;\l\in\CR
\end{equation}
satisfying for all $\l\in\CR$ the relations \eqref{1.14} and \eqref{1.15}. Moreover, in the case $n_+(\Tmi)=n_-(\Tmi)$ the equalities
\begin {gather}
C_a=(-C_{1a},\; i\wh C_0-\tfrac 1 2 \wh C_1, \; - C_{0a}):H\oplus\wh H\oplus
H\to \cH\label{4.23}\\
C_b=(C_{0b},\; -i\wh C_0-\tfrac 1 2 \wh C_1, \; C_{1b}):\cH_b\oplus\wh H\oplus
\cH_b\to \cH\label{4.24}
\end{gather}
give a bijective correspondence between all self-adjoint boundary parameters
$\tau$ of the form \eqref{2.4} and \eqref{4.7} and all operators
$(C_a,C_b):\bH\oplus\bH_b\to \cH$ satisfying \eqref{1.18}.
\end{lemma}
\begin{proof}
It follow from \eqref{3.33} that $\cH_2(=\cH_0\ominus\cH_1)=\cH_b^\perp$.
Therefore by \eqref{4.1} and \eqref{4.3}
\begin {gather*}
C_{01}(\l):=C_0(\l)\up \cH_1=(C_{0a}(\l),\; \wh C_0(\l), \;
C_{0b}(\l)\up\cH_b), \quad C_{02}(\l):=C_0(\l)\up
\cH_2= C_{0b}(\l)\up\cH_b^\perp\\
D_{01}(\l):=D_0(\l)\up \cH_1=(D_{0a}(\l),\; \wh D_0(\l), \;
D_{0b}(\l)\up\cH_b), \quad D_{02}(\l):=D_0(\l)\up \cH_2=
D_{0b}(\l)\up\cH_b^\perp
\end{gather*}
and the immediate calculations give
\begin {gather*}
2\im (C_1(\l)C_{01}^*(\l))+C_{02}(\l)
C_{02}^*(\l)=i (C_a(\l)J C_a^*(\l)-C_b(\l)J_b C_b^*(\l)), \quad \l\in\bC_+\\
2\im (D_1(\l)D_{01}^*(\l))+D_{02}(\l) D_{02}^*(\l)=i (D_a(\l)J
D_a^*(\l)-D_b(\l)J_b D_b^*(\l)), \quad \l\in\bC_-\\
C_1(\l)D_{01}^*(\ov\l)-C_{01}(\l)D_1^*(\ov\l)+i
C_{02}(\l)D_{02}^*(\ov\l)=C_b(\l)J_b D_b^*(\ov\l)-C_a(\l)J D_a^*(\ov\l), \quad
\l\in\bC_+.
\end{gather*}
Moreover, the following equivalences are obvious
\begin {gather*}
\ran (C_0(\l),\, C_1(\l))=\cH_0\iff \ran (C_a(\l),\, C_b(\l))=\cH_0\\
 \ran
(D_0(\l),\, D_1(\l))=\cH_1\iff \ran (D_a(\l),\, D_b(\l))=\cH_1.
\end{gather*}
This and \cite[Proposition 2.5]{AlbMalMog13} yield the desired statements.
\end{proof}
Let $\pair$ be a boundary parameter defined by \eqref{2.1} and
\eqref{4.1}--\eqref{4.4} and let $C_a(\l),\; C_b(\l)$ and $D_a(\l),\; D_b(\l)$
be the operator-functions \eqref{4.10}--\eqref{4.13} (hence the relations
\eqref{4.15}--\eqref{4.17} hold). For a given function $f\in\lI$ consider the
 boundary problem
\begin {gather}
J y'-B(t)y = \l \D(t) y+\D(t)f(t), \quad t\in\cI\label{4.28}\\
C_a(\l)y(a)+C_b(\l)\G_b y = 0, \quad \l\in\bC_+;\qquad D_a(\l)y(a)+D_b(\l)\G_b
y = 0, \quad \l\in\bC_-.\label{4.29}
\end{gather}
A function $y(\cd,\cd):\cI\tm (\CR)\to\bH$ is called a solution of this problem
if for each $\l\in\CR$ the function $y(\cd,\l)$ belongs to $\AC\cap\lI$ and
satisfies the equation \eqref{4.28} a.e. on $\cI$ (so that $y\in\dom\tma$) and
the boundary conditions \eqref{4.29}.

If $n_+(\Tmi)=n_-(\Tmi)$ and $\tau$ is a boundary parameter  defined by
\eqref{2.3} and \eqref{4.5}, \eqref{4.6}, then the boundary conditions
\eqref{4.29} take the form
\begin {gather}\label{4.30}
C_a(\l)y(a)+C_b(\l)\G_b y = 0, \quad \l\in\CR,
\end{gather}
with $C_a(\l)$ and $C_b(\l)$ given by \eqref{4.18} and \eqref{4.19}. Moreover,
if $\tau$ is a self-adjoint boundary parameter \eqref{2.4}, \eqref{4.7}, then
\eqref{4.30} becomes a self-adjoint boundary condition
\begin {gather}\label{4.31}
C_a y(a)+C_b \G_b y = 0,
\end{gather}
where $C_a$ and $C_b$ are the operators \eqref{4.23} and \eqref{4.24} (hence
they satisfy \eqref{1.18}).

In the following theorem we describe all generalized resolvents (and,
consequently, all exit space self-adjoint extensions) of $\Tmi$ in terms of
$\l$-depending boundary conditions.
\begin{theorem}\label{th4.3}
Let $\pair$ be a boundary parameter defined by \eqref{2.1} and
\eqref{4.1}--\eqref{4.4} and let $C_a(\l),\;C_b(\l)$ and $D_a(\l),\;D_b(\l)$ be
given by \eqref{4.10}--\eqref{4.13}. Then for every $f\in\lI$ the boundary
problem \eqref{4.28},  \eqref{4.29} has a unique solution $y(t,\l)=y_f(t,\l) $
and the equality
\begin {equation*}
R(\l)\wt f = \pi(y_f(\cd,\l)), \quad \wt f\in \LI, \quad f\in\wt f, \quad
\l\in\CR
\end{equation*}
defines a generalized resolvent $R(\l)=:R_\tau(\l)$ of $\Tmi$. Conversely, for
each generalized resolvent $R(\l)$ of $\Tmi$ there exists a unique boundary
parameter $\tau$ such that $R(\l)=R_\tau(\l)$.

If in addition $n_+(\Tmi)=n_-(\Tmi)$, then   the above statements hold with the
boundary parameter $\tau$ of the form \eqref{2.3}, \eqref{4.5}, \eqref{4.6} and
the boundary condition \eqref{4.30} in place of \eqref{4.29}. Moreover,
$R_\tau(\l)$ is a canonical resolvent of $\Tmi$ if and only if $\tau$ is a
self-adjoint boundary parameter \eqref{2.4}, \eqref{4.7}. In this case
$R_\tau(\l)=(\wt T^\tau - \l)^{-1}$, where $\wt T^\tau$  is given by \eqref{1.19} with the operators $C_a$ and $C_b$ of the form \eqref{4.23},  \eqref{4.24}.
\end{theorem}
\begin{proof}
Let $\Pi_+=\bta$ be the decomposing boundary triplet \eqref{3.33}--\eqref{3.35}
for $\Tma$. Then the immediate checking shows that
\begin {equation}\label {4.32.0}
C_0(\l)\G_0\{\wt y,\wt f\}-C_1(\l)\G_1\{\wt y,\wt f\}=C_a(\l)y(a)+
C_b(\l)\G_b(y), \quad \{\wt y,\wt f\}\in\Tma.
\end{equation}
Hence the boundary problem \eqref{4.28}, \eqref{4.29} is equivalent to the
 following one:
\begin {gather}
\{\wt y, \l \wt y+\wt f\}\in \Tma\label{4.32.1}\\
C_0(\l)\G_0\{\wt y, \l \wt y+\wt f\}-C_1(\l) \G_1 \{\wt y, \l \wt y+\wt f\}=0,
\quad \l \in\bC_+\label{4.32.2}\\
D_0(\l)\G_0\{\wt y, \l \wt y+\wt f\}-D_1(\l) \G_1 \{\wt y, \l \wt y+\wt f\}=0,
\quad \l \in\bC_-.\label{4.32.3}
\end{gather}
Now application of \cite[Theorem 3.11]{Mog13.2} gives the required statement.
\end{proof}
\subsection{Characteristic matrices}
The following theorem is well known (see e.g. \cite{Bru78,DLS93,Sht57}).
\begin{theorem}\label{th4.6}
Let $Y_0(\cd,\l)$ be the $[\bH]$-valued operator solution of Eq. \eqref{3.2}
satisfying $Y_0(a,\l)=I_{\bH}$. Then for each generalized resolvent $R(\l)$ of
$\Tmi$ there exists a unique operator function $\Om (\cd):\CR\to [\bH]$ such
that for each  $ \wt f\in\LI$ and $\l\in\CR$
\begin {equation*}
R(\l)\wt f=\pi\left(\int_\cI Y_0(\cd,\l)(\Om(\l)+\tfrac 1 2 \, {\rm
sgn}(t-x)J)Y_0^*(t,\ov\l)\D(t) f(t)\,dt \right), \quad f\in\wt f.
\end{equation*}
Moreover, $\Om(\cd)$ is a Nevanlinna operator function.
\end{theorem}
\begin{definition}\label{def4.6a}$\,$\cite{Bru78,Sht57}
The operator function $\Om(\cd) $ is called the characteristic matrix of the
symmetric system \eqref{3.1} corresponding  to the generalized resolvent
$R(\l)$.
\end{definition}
In the following the characteristic matrix $\Om(\cd)$ will be called canonical
if it corresponds to the canonical resolvent $R(\l)$ of $\Tmi$.

Since $\Om^*(\l)=\Om (\ov\l), \; \l\in\CR$, it follows that the characteristic
matrix $\Om(\cd)$ is uniquely defined, in fact, by its restriction onto
$\bC_+$.

Let the assumptions (A1) and (A2) from Subsection \ref{sub3.3} be satisfied,
let $\cH_0$ and $\cH_1$ be finite dimensional Hilbert spaces \eqref{3.33}, let
$\tau$ be a boundary parameter and let $R_\tau(\l)$ be the corresponding
generalized resolvent of $\Tmi$ (see Theorem \ref{th4.3}). In the following we
denote by $\Om_\tau(\cd)$ the characteristic matrix corresponding to
$R_\tau(\cd)$.

It follows from Theorem \ref{th4.3} that the equality $\Om(\l)=\Om_\tau(\l)$
gives a parametrization of all characteristic matrices of the system
\eqref{3.1} in terms of the boundary parameter $\tau$. In the following theorem
we represent this parametrization in the explicit form.
\begin{theorem}\label{th4.7}
Let $A_0$ be the maximal symmetric extension \eqref{3.36} of $\Tma$ and let  $M_+(\cd)$ be the operator function \eqref{3.43}--\eqref{3.47}. Moreover,
let $P_{\wh H}\in [H_0]$ be the orthoprojection in $H_0$ onto $\wh H$ (see
\eqref{3.0.1}) and let
\begin {gather}
\Om_0(\l)=\begin{pmatrix} m_0(\l) & -\tfrac 1 2  I_{H,H_0}\cr -\tfrac 1 2
P_{H_0,H} & 0\end{pmatrix}:\underbrace {H_0\oplus H}_{\bH}\to \underbrace{
H_0\oplus H}_{\bH}, \quad \l\in\CR\label{4.43}\\
S_1(\l)=\begin{pmatrix} m_0(\l)-\tfrac i 2 P_{\wh H} & M_{2+}(\l) \cr
-P_{H_0,H} & 0 \end{pmatrix}:\underbrace{H_0\oplus
\wt\cH_b}_{\cH_0}\to\underbrace{ H_0\oplus H}_{\bH},\quad \l\in\bC_+ \label{4.44}\\
S_2(\l)=\begin{pmatrix} m_0(\l)+\tfrac i 2 P_{\wh H} & -I_{H,H_0}\cr M_{3+}(\l)
& 0\end{pmatrix}:\underbrace{ H_0\oplus H}_{\bH}\to\underbrace{H_0\oplus
\cH_b}_{\cH_1}, \quad \l\in\bC_+ .\label{4.45}
\end{gather}
Then: {\rm (1)} $\Om_0(\cd)$ is the characteristic matrix corresponding to the generalized resolvent $R(\l)=(A_0-\l)^{-1}, \; \l\in\bC_+,$ of $\Tmi$.

{\rm (2)} For each boundary parameter $\pair$ of the form \eqref{2.1} the
operator $C_0 (\l)- C_1(\l) M_+(\l), \; \l\in\bC_+,$ is boundedly invertible;

{\rm (3)} The equality
\begin {equation}\label{4.46}
\Om(\l)=\Om_\tau(\l)=\Om_0(\l)+S_1(\l)(C_0 (\l)- C_1 M_+(\l))^{-1}
C_1(\l)S_2(\l), \quad \l\in\bC_+
\end{equation}
establishes a bijective correspondence between all boundary parameters $\pair$
defined by \eqref{2.1} and all characteristic matrices $\Om(\cd)$ of the system
\eqref{3.1}.
\end{theorem}
\begin{proof}
Let $\pair$ be a boundary parameter \eqref{2.1}. Since by Proposition
\ref{pr3.6} $M_+(\cd)$ is the Weyl function of the decomposing boundary triplet
$\Pi=\bta$ for $\Tma$, it follows from \cite[Theorem 3.11]{Mog13.2} that
$(\tau_+(\l)+M_+(\l))^{-1}\in [\cH_1,\cH_0], \; \l\in\bC_+$. Hence by
\cite[Lemma 2.1]{MalMog02} the operator $C_0 (\l)- C_1(\l) M_+(\l)$ is
boundedly invertible and
\begin {equation}\label{4.47}
T_\tau(\l):=-(\tau_+(\l)+M_+(\l))^{-1}=(C_0 (\l)- C_1(\l) M_+(\l))^{-1}C_1(\l),
\quad \l\in\bC_+.
\end{equation}
Next assume that $A_0(=\ker\G_0)$ is the extension \eqref{3.36} of $\Tmi$ and
that $\g_\pm(\cd) $ are the $\g$-fields  of the triplet $\Pi_+$. As it was
mentioned in the proof of Theorem \ref{th4.3} the generalized resolvent
$R_\tau(\l)$ is generated in fact by the boundary problem
\eqref{4.32.1}--\eqref{4.32.3}. Therefore according to \cite[Theorem
3.11]{Mog13.2} the following Krein formula for generalized  resolvents holds:
\begin {equation}\label{4.54}
R_\tau(\l)=(A_0-\l)^{-1}+\g_+(\l)T_\tau(\l)\g_-^*(\ov\l), \quad \l\in\bC_+.
\end{equation}
Let us show that for each $ \wt f\in\LI$ and $\l\in\bC_+$
\begin {equation}\label{4.55}
(A_0-\l)^{-1}\wt f =\pi\left(\int_\cI Y_0(\cd,\l)(\Om_0(\l)+\tfrac 1 2 \, {\rm
sgn}(t-x)J)Y_0^*(t,\ov\l)\D(t) f(t)\,dt \right), \quad f\in\wt f.
\end{equation}
It follows from \cite[Theorem 6.2]{AlbMalMog13} that  the equality
\begin {equation}\label{4.56}
(A_0-\l)^{-1}\wt f=\pi\left(\int_\cI  G_0 (\cd,t,\l)\D(t)f(t)\, dt \right ),
\quad \wt f\in\LI, \quad  f\in\wt f,\;\;\;\; \l\in\bC_+
\end{equation}
holds with the Green function $G_0(\cd,\cd,\l)$ of the form
\begin {equation} \label{4.57}
G_0 (x,t,\l)=\begin{cases} v_0(x,\l)\, \f^*(t,\ov\l), \;\; x>t \cr \f(x,\l)\,
v_0^* (t,\ov\l), \;\; x<t \end{cases}, \quad \l\in \bC_+.
\end{equation}
Here $\f(\cd,\l)$ is the $[H_0,\bH]$-valued operator solution of Eq.
\eqref{3.2} satisfying
\begin {equation*}
\f(a,\l)=\begin{pmatrix} I_{H_0} \cr 0 \end{pmatrix}:H_0\to H_0\oplus H,\quad
\l\in\bC,
\end{equation*}
and $v_0(\cd,\l)\in\lo{H_0}$ is the operator solution from Proposition
\ref{pr3.5}. Let
\begin {gather}
Y_+(t,\l):=(\f(t,\l), \, 0):H_0\oplus \cH_b\to\bH, \;\;\; \l\in\bC_+\label{4.58a}\\
Y_-(t,\l):=(\f(t,\l), \, 0):H_0\oplus \wt\cH_b\to\bH, \;\;\;
\l\in\bC_-\label{4.58b}
\end{gather}
and let $Z_\pm(t,\l)$ be given by \eqref{3.40} and \eqref{3.41}. Then
\eqref{4.57} can be represented as
\begin {equation} \label{4.59}
G_0 (x,t,\l)=\begin{cases} Z_+(x,\l)\, Y_-^*(t,\ov\l), \;\; x>t \cr Y_+(x,\l)\,
Z_-^* (t,\ov\l), \;\; x<t \end{cases}, \quad \l\in \bC_+.
\end{equation}
Since
\begin {equation*}
Z_\pm(a,\l)=\begin{pmatrix}(\cP_0+\wh\cP)v_0(a,\l) & (\cP_0+\wh\cP)u_\pm(a,\l)
\cr \cP_1 v_0(a,\l)& \cP_1u_\pm(a,\l)\end{pmatrix}, \quad \l\in\bC_\pm,
\end{equation*}
it follows from \eqref{3.45}, \eqref{3.49} and the first equalities in
\eqref{3.37} and \eqref{3.38} that
\begin {gather}
Z_+(a,\l)=\begin{pmatrix}m_0(\l)- \tfrac i 2 P_{\wh H} & M_{2+}(\l) \cr -
P_{H_0,H} & 0\end{pmatrix}: H_0\oplus\wt\cH_b\to H_0\oplus H, \quad
\l\in\bC_+\label{4.59a}\\
Z_-(a,\l)=\begin{pmatrix}m_0^*(\ov\l)- \tfrac i 2 P_{\wh H} & M_{3+}^*(\ov\l)
\cr - P_{H_0,H} & 0\end{pmatrix}: H_0\oplus\cH_b\to H_0\oplus H, \quad
\l\in\bC_-.\nonumber
\end{gather}
Therefore
\begin {equation} \label{4.59b}
Z_+(a,\l)= S_1(\l),  \qquad Z_-(a,\l)= S_2^*(\ov\l)
\end{equation}
and \eqref{3.3} yields
\begin {equation} \label{4.60}
Z_+(t,\l)=Y_0(t,\l) S_1(\l),\quad \l\in\bC_+;  \qquad Z_-(t,\l)=
Y_0(t,\l)S_2^*(\ov\l), \quad \l \in\bC_-.
\end{equation}
Moreover, by \eqref{4.58a} and  \eqref{4.58b}
\begin {equation*}
Y_+(a,\l)=\begin{pmatrix} I_{H_0} & 0 \cr 0 & 0
\end{pmatrix}: H_0\oplus\cH_b \to H_0\oplus H, \quad Y_-(a,\l)=\begin{pmatrix}
I_{H_0} & 0 \cr 0 & 0\end{pmatrix}: H_0\oplus\wt\cH_b \to H_0\oplus H
\end{equation*}
and \eqref{3.3} gives $Y_+(t,\l)=Y_0(t,\l) Y_+(a,\l)$ and $ Y_-(t,\l)=
Y_0(t,\l) Y_-(a,\l)$. This and  \eqref{4.59} imply that
\begin {equation} \label{4.61}
G_0 (x,t,\l)=\begin{cases} Y_0(x,\l)(S_1(\l)Y_-^*(a,\ov\l))Y_0^*(t,\ov\l), \;\;
x>t \\  Y_0(x,\l)(Y_+(a,\l)S_2(\l))Y_0^*(t,\ov\l), \;\; x<t
\end{cases}, \quad \l\in \bC_+.
\end{equation}
Observe also that  the operator $J$ can be represented as
\begin {equation} \label{4.61a}
J=\begin{pmatrix}iP_{\wh H} & -I_{H,H_0}\cr P_{H_0,H} & 0
\end{pmatrix}:H_0\oplus H\to H_0\oplus H.
\end{equation}
and the direct calculations with taking \eqref{4.43} into account give
\begin {equation*}
S_1(\l)Y_-^*(a,\ov\l)=\Om_0(\l)-\tfrac 1 2 J, \quad
Y_+(a,\l)S_2(\l)=\Om_0(\l)+\tfrac 1 2 J, \;\;\;\l\in\bC_+.
\end{equation*}
Combining these equalities with \eqref{4.61} and \eqref{4.56} one gets
\eqref{4.55}. Hence statement (1) holds.

Next in view of \eqref{3.42} and \eqref{4.60} $\g_-(\ov\l)=\pi Z_-(\ov\l)$.
Therefore, by Lemma \ref{lem3.1b} and the second equality in \eqref{4.60}, for
each $\wt f\in\LI$ and $\l\in\bC_+$ one has
\begin {equation*}
\g_-^*(\ov\l)\wt f=\int_\cI Z_-^*(t,\ov\l) \D(t)f(t)\,dt=\int_\cI
S_2(\l)Y_0^*(t,\ov\l)\D(t)f(t)\,dt, \quad f\in\wt f.
\end{equation*}
This and the first equalities in \eqref{3.42} and \eqref{4.60} imply that for
each $\wt f\in\LI$ and $\l\in\bC_+$
\begin {equation} \label{4.62}
\g_+(\l)T_\tau(\l)\g_-^*(\ov\l)\wt f= \pi \left( \int_\cI Y_0(\cd,\l)S_1(\l)
T_\tau(\l)S_2(\l)Y_0^*(t,\ov\l)\D(t)f(t)\,dt \right), \quad f\in\wt f.
\end{equation}
Now combining \eqref{4.54} with \eqref{4.55} and \eqref{4.62} we obtain  the
equality
\begin {equation*}
R_\tau(\l)\wt f=\pi\left(\int_\cI Y_0(\cd,\l)(\Om_\tau(\l)+\tfrac 1 2 \, {\rm
sgn}(t-x)J)Y_0^*(t,\ov\l)\D(t) f(t)\,dt \right), \;\; \wt
f\in\LI,\;\;\l\in\bC_+,
\end{equation*}
where $\Om_\tau(\cd)$ is the operator function \eqref{4.46}. Thus
$\Om_\tau(\cd)$ is the characteristic matrix of the generalized resolvent
$R_\tau(\l)$, which in view of Theorem \ref{th4.3} yields statement (3) of the theorem.
\end{proof}
Let as before $M_+(\l), \; \l\in\bC_+,$ be given by \eqref{3.43}--\eqref{3.47}.
With each boundary parameter $\pair$ of the form  \eqref{2.1} we associate a
holomorphic operator function $\wt\Om_{\tau }(\cd):\bC_+\to [\cH_0\oplus\cH_1,
\cH_1\oplus\cH_0] $ given by
\begin{gather}
\wt\Om_{\tau}(\l)=\begin{pmatrix}\wt\om_{1}(\l) & \wt\om_{2}(\l) \cr
\wt\om_{3}(\l) & \wt\om_{4}(\l) \end{pmatrix}:
\cH_0\oplus\cH_1\to \cH_1\oplus\cH_0, \quad \l\in\bC_+\label{4.65}\\
\wt\om_{1}(\l)=M_+(\l)-M_+(\l)(\tau_+(\l) +M_+(\l))^{-1}M_+(\l)
 \label{4.66}\\
\wt\om_{2}(\l)= -\tfrac 1 2 I_{\cH_1} +M_+(\l) (\tau_+(\l) +M_+(\l))^{-1}
\label{4.67}\\
\wt\om_{3}(\l)=-\tfrac 1 2 I_{\cH_0} +(\tau_+(\l) +M_+(\l))^{-1}M_+(\l),\qquad
\wt\om_{4}(\l)=-(\tau_+(\l) + M_+(\l))^{-1}\label{4.69}
\end{gather}
It follows from \eqref{4.47} that the equalities \eqref{4.66}--\eqref{4.69} can
be represented as
\begin{gather}
\wt\om_{1}(\l)=M_+(\l)(C_0(\l)-C_1(\l)M_+(\l))^{-1} C_0(\l) \label{4.70}\\
\wt\om_{2}(\l)=-\tfrac 1 2 I_{\cH_1} -M_+(\l) (C_0(\l)-C_1(\l)M_+(\l))^{-1}
C_1(\l)\label{4.71}\\
\wt\om_{3}(\l)=\tfrac 1 2 I_{\cH_0}-(C_0(\l)-C_1(\l)M_+(\l))^{-1} C_0(\l),
\label{4.72}\\
\wt\om_{4}(\l)=(C_0(\l)-C_1(\l)M_+(\l))^{-1} C_1(\l).\label{4.73}
\end{gather}

In the following proposition we give a somewhat other parametrization of all
characteristic matrices $\Om(\l)$ (cf. \eqref{4.46}).
\begin{proposition}\label{pr4.9}
Let $P_{\wh H}\in [H_0]$ be the orthoprojection  in $H_0$ onto $\wh H$ and let
\begin{gather}
X_1=\begin{pmatrix}P_{\cH_1,H_0} & \tfrac i 2 P_{\wh H} P_{\cH_0,H_0} \cr 0 &
P_{\cH_0,H}\end{pmatrix}: \cH_1\oplus\cH_0\to\underbrace{H_0\oplus
H}_{\bH}\label{4.74}\\
X_2=\begin{pmatrix}P_{\cH_0,H_0} & \tfrac i 2 P_{\wh H} P_{\cH_1,H_0} \cr 0 &
P_{\cH_1,H}\end{pmatrix}: \cH_0\oplus\cH_1\to\underbrace{H_0\oplus
H}_{\bH}\label{4.75}
\end{gather}
(clearly the operators $P_{\cH_j,H_0}$ and  $P_{\cH_j,H}$ make sense, because
in view of \eqref{3.33} $H\subset H_0\subset \cH_j, \; j\in\{0,1\}$). Then for
each boundary parameter $\pair$ the corresponding characteristic matrix
$\Om(\l)=\Om_\tau(\l)$ of the system \eqref{3.1} admits the representation
\begin {equation} \label{4.76}
\Om_\tau(\l)=X_1 \wt \Om_\tau(\l)X_2^*, \quad \l\in\bC_+.
\end{equation}
\end{proposition}
\begin{proof}
Let $T_\tau(\l)$ be given by \eqref{4.47}. Since
\begin {equation} \label{4.77}
X_2^*=\begin{pmatrix}I_{H_0,\cH_0} & 0 \cr -\tfrac i 2 I_{H_0,\cH_1}P_{\wh H} &
I_{H,\cH_1} \end{pmatrix}:H_0\oplus H\to\cH_0\oplus\cH_1,
\end{equation}
it follows from \eqref{4.74} and \eqref{4.65}--\eqref{4.69} that
\begin {equation} \label{4.78}
X_1 \wt \Om_\tau(\l)X_2^*=\begin{pmatrix}\om_{1}(\l) & \om_{2}(\l) \cr
\om_{3}(\l) & \om_{4}(\l) \end{pmatrix}: H_0\oplus H\to H_0\oplus H, \quad
\l\in\bC_+,
\end{equation}
where
\begin{gather*}
\om_1(\l)=m_0(\l)+P_{\cH_1,H_0}M_+(\l)T_\tau(\l)M_+(\l)\up H_0+ \tfrac i
2P_{\cH_1,H_0}M_+(\l)T_\tau(\l)\up H_0\cd P_{\wh H}- \\
\tfrac i 2 P_{\wh H}P_{\cH_0,H_0}T_\tau(\l)M_+(\l)\up H_0+
 \tfrac 1 4 P_{\wh
H}P_{\cH_0,H_0}T_\tau(\l)\up H_0\cd P_{\wh H} \\
\om_2(\l)=-\tfrac 1 2 I_{H,H_0}-P_{\cH_1,H_0}M_+(\l)T_\tau(\l)\up H+\tfrac i 2
P_{\wh H}P_{\cH_0,H_0}T_\tau(\l)\up H\nonumber\\
\om_3(\l)=-\tfrac 1 2 P_{H_0,H}-P_{\cH_0,H}T_\tau(\l)M_+(\l)\up H_0-\tfrac i 2
P_{\cH_0,H}T_\tau(\l)\up H_0\cd P_{\wh H}\nonumber\\
\om_4(\l)=P_{\cH_0,H}T_\tau(\l)\up H\nonumber
\end{gather*}
(in the  equality for $\om_1(\l)$ we made use of the relation
$m_0(\l)=P_{\cH_1,H_0} M_+(\l)\up H_0$ implied by \eqref{3.43}). Next, in view
of \eqref{3.43} the equalities \eqref{4.44} and \eqref{4.45} can be written as
\begin{gather}
S_1(\l)=\begin{pmatrix}P_{\cH_1,H_0}M_+(\l)-\tfrac i 2 P_{\wh H}
P_{\cH_0,H_0}\cr - P_{\cH_0,H}\end{pmatrix}:\cH_0\to H_0\oplus H\label{4.79.0}\\
S_2(\l)=(M_+(\l)\up H_0+\tfrac i 2 I_{H_0,\cH_1} P_{\wh H}, \;
-I_{H,\cH_1}):H_0\oplus H\to\cH_1\label{4.79a}
\end{gather}
This and \eqref{4.46} yield $ \Om_\tau(\l)=\begin{pmatrix} m_0(\l) & -\tfrac 1
2  I_{H,H_0}\cr -\tfrac 1 2 P_{H_0,H} & 0\end{pmatrix}+$

$ +\begin{pmatrix}P_{\cH_1,H_0}M_+(\l)-\tfrac i 2 P_{\wh H} P_{\cH_0,H_0}\cr -
P_{\cH_0,H}\end{pmatrix} \cd T_\tau(\l)\cd (M_+(\l)\up H_0+\tfrac i 2
I_{H_0,\cH_1} P_{\wh H}, \; -I_{H,\cH_1})$

 and the immediate calculations show
that
\begin {equation} \label{4.80}
\Om_{\tau}(\l)=\begin{pmatrix}\om_{1}(\l) & \om_{2}(\l) \cr \om_{3}(\l) &
\om_{4}(\l) \end{pmatrix}: H_0\oplus H\to H_0\oplus H, \quad \l\in\bC_+
\end{equation}
Now comparing \eqref{4.78} and \eqref{4.80} we arrive at the equality
\eqref{4.76}.
\end{proof}
\begin{theorem}\label{th4.10}
Assume the hypotheses of Lemma \ref{lem4.2}. Moreover, let $\pair$ be a
boundary parameter defined by \eqref{2.1} and \eqref{4.1}--\eqref{4.4} and let
$C_a(\l)$ and $C_b(\l)$ be given by \eqref{4.10} and \eqref{4.11}. Then:

{\rm (1)} For each $\l\in\bC_+$ there exists a unique operator solution
$Z_\tau(\cd,\l)\in\cL_\D^2 [\bH]$ of Eq. \eqref{3.2} satisfying the boundary
condition
\begin {equation} \label{4.81}
C_a (\l)(Z_\tau (a,\l)+J)+C_b(\l)\G_b Z_\tau(\l)=0, \quad \l\in\bC_+
\end{equation}
(here $Z_\tau(\l)$ is the mapping \eqref{3.10} for $Z_\tau(\cd,\l)$).

{\rm (2)} The corresponding characteristic matrix $\Om_\tau(\cd)$ satisfies
\begin {equation} \label{4.82}
\Om_\tau(\l)=Z_\tau(a,\l) + \tfrac 1 2 J, \quad \l\in\bC_+.
\end{equation}

{\rm (3)} The following inequality holds
\begin {equation} \label{4.83}
(\im \l)^{-1}\cd \im \Om_\tau(\l)\geq \int_\cI Z_\tau^*(t,\l) \D(t)
Z_\tau(t,\l) \, dt, \quad  \l\in\bC_+.
\end{equation}
\end{theorem}
\begin{proof}
(1) Let $\Pi_+=\bta$ be the decomposing boundary triplet
\eqref{3.33}--\eqref{3.35} for $\Tma$ and let $M_+(\cd)$ and $\g_+(\cd)$ be the
Weyl function and the $\g$-field of $\Pi_+$ respectively. Moreover, let
$Z_+(\cd,\l)\in\lo{\cH_0}$ be the operator solution of Eq. \eqref{3.2} defined
in Proposition \ref{pr3.6} and let
\begin {equation} \label{4.84}
Z_\tau(t,\l):=-Z_+(t,\l) (C_0(\l)-C_1(\l)M_+(\l))^{-1}C_a(\l)J, \quad
\l\in\bC_+.
\end{equation}
Clearly, $Z_\tau(\cd,\l)\in\cL_\D^2[\bH]$ and $Z_\tau(\cd,\l)$ is an operator
solution of Eq. \eqref{3.2}. Let us show that $Z_\tau(\cd,\l)$ satisfies
\eqref{4.81}. Assume that $h\in\bH$, $h_0:=-(C_0(\l)-C_1(\l)M_+(\l))^{-1}
C_a(\l)Jh$ and
\begin {equation} \label{4.85}
 \wt y=\pi (Z_\tau(\cd,\l)h).
\end{equation}
 Then by \eqref{4.84} $\wt y=\pi Z_+(\l)h_0$ and the
equality \eqref{3.42} yields
\begin {equation} \label{4.86}
\wt y=\g_+(\l)h_0.
\end{equation}
Combining \eqref{4.85} with \eqref{4.32.0} one gets
\begin {equation} \label{4.87}
C_0(\l)\G_0\{\wt y,\l \wt y\}-C_1(\l)\G_1\{\wt y,\l \wt
y\}=(C_a(\l)Z_\tau(a,\l)+C_b(\l)\G_b Z_\tau(\l))h.
\end{equation}
On the other hand, combining of \eqref{4.86} with \eqref{2.10} and \eqref{2.8}
yields $\G_0 \{\wt y,\l \wt y\}=h_0$ and $\G_1\{\wt y,\l \wt y\}=M_+(\l)h_0$.
Therefore
\begin {equation*}
C_0(\l)\G_0\{\wt y,\l \wt y\}-C_1(\l)\G_1\{\wt y,\l \wt y\}=(C_0(\l)-
C_1(\l)M_+(\l))h_0=-C_a(\l)J h.
\end{equation*}
Comparing this equality with \eqref{4.87} one obtains
\begin {equation*}
C_a(\l)Z_\tau(a,\l)+C_b(\l)\G_b Z_\tau(\l)=-C_a(\l)J,\quad \l\in\bC_+.
\end{equation*}
This implies \eqref{4.81}.

To prove uniqueness of $Z_\tau(\cd,\l)$ assume that $\wt Z_\tau(\cd,\l)
\in\cL_\D^2[\bH] $ is another  solution of Eq. \eqref{3.2} satisfying
\eqref{4.81}. Then for each $h\in\bH$ the function $y=(Z_\tau(t,\l)-\wt
Z_\tau(t,\l))h$ is a solution of the homogeneous boundary problem \eqref{4.28},
\eqref{4.29} (with $f=0$). Since by Theorem \ref{th4.3}  such a problem has a
unique solution $y=0$, it follows that $Z_\tau(t,\l)=\wt Z_\tau(t,\l)$.

(2) Assume that $S_2(\l)$ is given by \eqref{4.45} and that
\begin {equation} \label{4.87a}
Z_0(t,\l):=(Z_+(t,\l)\up H_0, \; 0):H_0\oplus H\to \bH, \quad \l\in\bC_+.
\end{equation}
Then by \eqref{4.59a}
\begin {equation*}
Z_0(a,\l)=\begin{pmatrix} m_0(\l) - \tfrac i 2 P_{\wh H} & 0 \cr -P_{H_0,H} &
0\end{pmatrix}:H_0\oplus H\to H_0\oplus H
\end{equation*}
and the equalities \eqref{4.43} and \eqref{4.61a} yield
\begin {equation} \label{4.87b}
Z_0(a,\l)=\Om_0(\l)- \tfrac 1 2 J, \quad \l\in\bC_+.
\end{equation}
Next we show that
\begin {equation} \label{4.88}
 Z_\tau(t,\l)=Z_0(t,\l)+Z_+(t,\l)(C_0(\l)-C_1(\l)M_+(\l))^{-1}C_1(\l)S_2(\l),
 \quad \l\in\bC_+.
\end{equation}
Since by \eqref{4.10}
\begin {equation*}
C_a(\l)J=(-C_{0a}(\l), \; -\wh C_0(\l)-\tfrac i 2 \wh C_1(\l), \; C_{1a}(\l)):
H\oplus\wh H\oplus H\to \cH_0,
\end{equation*}
it follows from \eqref{4.1} and \eqref{4.2} that
\begin {equation}\label{4.89}
C_a(\l)Jh=-C_0(\l)(\cP_0 h+\wh\cP h)-\tfrac i 2 C_1(\l)\wh\cP h+C_1(\l)\cP_1h,
 \quad h\in\bH.
\end{equation}
Let $P_{\bH,H_0}(\in [\bH,H_0])$ be the orthoprojection in $\bH$ onto $H_0$
(see \eqref{3.0.1}) and let as before $P_{\wh H}(\in [H_0])$ be the
orthoprojection in $H_0$ onto $\wh H$. Then $\cP_0 h+\wh\cP h=P_{\bH,H_0}h,\; \wh\cP h
=P_{\wh H}P_{\bH,H_0}h$ and the equality \eqref{4.89} can be written as
\begin {equation*}
C_a(\l)Jh=-C_0(\l)P_{\bH,H_0} h -\tfrac i 2 C_1(\l) P_{\wh H}P_{\bH,H_0}h+
C_1(\l)\cP_1h,  \quad h\in\bH.
\end{equation*}
Moreover, by \eqref{4.79a}
\begin {equation*}
S_2(\l)h=M_+(\l)P_{\bH,H_0} h +\tfrac i 2  P_{\wh H}P_{\bH,H_0}h-\cP_1h,  \quad h\in\bH,
\end{equation*}
and combining of the last two equalities yields
\begin {equation*}
C_a(\l)Jh=-(C_0(\l)-C_1(\l)M_+(\l))P_{\bH,H_0}h-C_1(\l)S_2(\l)h, \quad h\in\bH.
\end{equation*}
This and  \eqref{4.84} imply that for each $h\in\bH$
\begin {equation}\label{4.89a}
Z_\tau(t,\l)h=Z_+ (t,\l)P_{\bH,H_0}h+Z_+ (t,\l)(C_0(\l)-C_1(\l)M_+(\l))^{-1} C_1(\l)S_2(\l)h.
\end{equation}
Since by \eqref{4.87a} $Z_+ (t,\l)P_{\bH,H_0}h=Z_0(t,\l)h$, it follows from \eqref{4.89a} that $Z_\tau(t,\l)$ admits the representation \eqref{4.88}.

Now combining \eqref{4.88} with \eqref{4.87b} and the first equality in \eqref{4.59b} and then taking \eqref{4.46} into account one obtains the equality \eqref{4.82}.

(3) Let us show that
\begin {equation}\label{4.90}
\Om_0(\mu)-\Om_0^*(\l)=(\mu-\ov\l)\g_1^*(\l)\g_1(\mu), \;\;\;\; S_1(\mu)-S_2^*(\l)P_1=(\mu-\ov\l)\g_1^*(\l)\g_+(\mu),
\end{equation}
where $\mu,\l\in\bC_+, \; P_1=P_{\cH_0,\cH_1}$  is the orthoprojection in $\cH_0$ onto $\cH_1$ and
\begin {equation}\label{4.90a}
\g_1(\l)=(\g_+(\l)\up H_0,\; 0):H_0\oplus H\to \LI.
\end{equation}
The first equality in \eqref{4.90} is immediate from \eqref{2.9}. Next, by \eqref{4.79a} one has
\begin {equation*}
S_2^*(\l)P_1=\begin{pmatrix}P_{\cH_0,H_0}M_+^*(\l)-\tfrac i 2 P_{\wh H}P_{\cH_1,H_0} \cr - P_{\cH_1,H}\end{pmatrix}P_1= \begin{pmatrix}P_{\cH_0,H_0}M_+^*(\l)P_1-\tfrac i 2 P_{\wh H}P_{\cH_0,H_0} \cr - P_{\cH_0,H}\end{pmatrix},
\end{equation*}
which in view of \eqref{4.79.0} yields
\begin {equation}\label{4.91}
S_1(\mu)-S_2^*(\l)P_1=\begin{pmatrix}P_{\cH_1,H_0}M_+(\mu)-P_{\cH_0,H_0}M_+^*(\l)
P_1 \cr 0 \end{pmatrix}:\cH_0\to H_0\oplus H, \;\; \mu,\l\in\bC_+.
\end{equation}
Since $H_0\subset \cH_1$, it follows that $H_0\perp \cH_2$ and hence $P_{\cH_0,H_0}P_2=0$. Therefore application of the operator $P_{\cH_0,H_0}$ to the identity \eqref{2.9} yields
\begin {equation*}
P_{\cH_1,H_0}M_+(\mu)-P_{\cH_0,H_0} M_+^*(\l)P_1=(\mu-\ov\l)P_{\cH_0,H_0} \g_+^*(\l)\g_+(\mu), \quad \mu,\l\in\bC_+.
\end{equation*}
Combining this equality with \eqref{4.91} one gets the second equality in \eqref{4.90}.

Now application of \cite[Lemma 21]{Mog11} to the representation \eqref{4.46} of $\Om_\tau(\cd)$ with taking \eqref{4.90} and \eqref{2.9} into account yields
\begin {equation}\label{4.92}
\im \Om_\tau(\l)\geq \im\l\cd \g_\tau^*(\l)\g_\tau(\l), \quad \l\in\bC_+,
\end{equation}
where
\begin {equation}\label{4.93}
\g_\tau(\l)=\g_1(\l)+\g_+(\l)(C_0(\l)-C_1(\l)M_+(\l))^{-1} C_1(\l)S_2(\l), \quad \l\in\bC_+.
\end{equation}
Since by \eqref{3.42} $\g_+(\l)=\pi Z_+(\l)$, it follows from \eqref{4.90a} and \eqref{4.87a} that $\g_1(\l)=\pi Z_0(\l)$. Combining these equalities with \eqref{4.93} and \eqref{4.88} we obtain
\begin {equation}\label{4.94}
\g_\tau(\l)=\pi Z_\tau (\l),
\end{equation}
where $Z_\tau(\l)$ is the mapping \eqref{3.10} for $Z_\tau(\cd,\l)$. Therefore by Lemma \ref{lem3.1b}
\begin {equation}\label{4.95}
\g_\tau^*(\l)\wt f=\int_\cI Z_\tau^*(t,\l)\D(t)f(t)\,dt,\quad \wt f\in\LI,\;\;f\in\wt f,
\end{equation}
and combining  of \eqref{4.92} with \eqref{4.94} and \eqref{4.95}  gives \eqref{4.83}.
\end{proof}
\subsection{The case of equal deficiency indices}
In the case of equal deficiency indices $n_+(\Tmi)=n_-(\Tmi)$ the above results can be rather simplified. Namely, the following theorems are immediate from Theorems \ref{th4.7}, \ref{th4.10} and Proposition \ref{pr4.9}.
\begin{theorem}\label{th4.11}
Let $n_+(\Tmi)=n_-(\Tmi)$ (so that $\wt\cH_b=\cH_b$), let $\cH=H_0\oplus\cH_b$, let $A_0=A_0^*$ be given by \eqref{3.36} and let $M(\cd)$ be the (Nevanlinna) operator function defined by  \eqref{3.50}--\eqref{3.52}. Moreover, let $\Om_0(\cd)$ be the operator function \eqref{4.43} and let
\begin {equation}\label{4.96}
S(\l)=\begin{pmatrix} m_0(\l)-\tfrac i 2 P_{\wh H} & M_{2}(\l) \cr
-P_{H_0,H} & 0 \end{pmatrix}:\underbrace{H_0\oplus
\cH_b}_{\cH}\to\underbrace{ H_0\oplus H}_{\bH},\quad \l\in\CR.
\end{equation}
Then: {\rm (1)} $\Om_0(\cd)$ is the characteristic matrix corresponding to the canonical resolvent $(A_0-\l)^{-1}$ and the equality
\begin {equation*}
\Om(\l)=\Om_\tau(\l)=\Om_0(\l)+S(\l)(C_0 (\l)- C_1 M(\l))^{-1}
C_1(\l)S^*(\ov\l), \quad \l\in\CR
\end{equation*}
establishes a bijective correspondence between all boundary parameters $\tau$
 of  the form \eqref{2.3} and all characteristic matrices $\Om(\cd)$ of the system \eqref{3.1}. Moreover, the characteristic matrix $\Om(\cd)=\Om_\tau(\cd)$ is canonical if and only if the boundary parameter $\tau$ is self-adjoint.

{\rm (2)} For each boundary parameter $\tau=\tau(\l)$  the corresponding characteristic matrix $\Om(\l)=\Om_\tau(\l)$ of the system  \eqref{3.1} admits the representation
\begin {equation*}
\Om_\tau(\l)=X \wt \Om_\tau(\l)X^*, \quad \l\in\CR,
\end{equation*}
with the Nevanlinna operator function $\wt \Om_\tau(\cd): \CR\to [\cH\oplus\cH]) $ of the form \eqref{1.25} and the operator $X\in [\cH\oplus\cH,\bH]$ given by
\begin{gather}\label{4.98}
X=\begin{pmatrix}P_{\cH,H_0} & \tfrac i 2 P_{\wh H} P_{\cH,H_0} \cr 0 &
P_{\cH,H}\end{pmatrix}: \cH\oplus\cH\to\underbrace{H_0\oplus
H}_{\bH}
\end{gather}
\end{theorem}
\begin{theorem}\label{th4.12}
Let $n_+(\Tmi)=n_-(\Tmi)$, let $\bH_b=\cH_b\oplus\wh H\oplus\cH_b$ and let $J_b$ be the operator \eqref{1.12}. Moreover, let $\tau$ be a boundary parameter defined by  \eqref{2.3} and \eqref{4.5}, \eqref{4.6} and let $C_a(\l)$ and $C_b(\l)$ be the operator functions \eqref{4.18} and \eqref{4.19}. Then for each $\l\in\CR$ there exists a unique operator solution
$Z_\tau(\cd,\l)\in\cL_\D^2 [\bH]$ of Eq. \eqref{3.2} satisfying \eqref{4.81} (with $\l\in\CR$). Moreover, $\Om_\tau(\l)=Z_\tau(a,\l) + \frac 1 2 J, \; \l\in\CR,$ and the inequality \eqref{4.83} is valid for all $\l\in\CR$.

If in addition $\tau$ is a self-adjoint boundary parameter, then the following identity holds:
\begin {equation*}
\Om_\tau(\mu)-\Om_\tau^*(\l)=(\mu-\ov\l)\int_\cI Z_\tau^*(t,\l) \D(t)
Z_\tau(t,\mu) \, dt, \quad  \mu,\l\in\CR.
\end{equation*}
This implies that for the canonical characteristic matrix $\Om_\tau(\cd)$ the inequality \eqref{4.83} turns into the equality, which holds for all $\l\in\CR$.
\end{theorem}

\end{document}